\documentclass[11pt]{amsart}

\usepackage[left=2cm,top=2cm,right=2cm]{geometry}

\usepackage{parskip}

\usepackage{graphicx}

\usepackage{ulem}

\usepackage{bbm}

\usepackage[usenames,dvipsnames,svgnames,table]{xcolor}

\usepackage{graphicx, color}
\usepackage{psfrag}
\usepackage{amsmath, amsthm,amssymb}
\usepackage{amsfonts,dsfont}
\usepackage{hyperref}

\newtheorem{thm}{Theorem}[section]
\newtheorem{ass}[thm]{Assumption}
\newtheorem{coro}[thm]{Corollary}
\newtheorem{lem}[thm]{Lemma}
\newtheorem{prop}[thm]{Proposition}

\theoremstyle{definition}

\theoremstyle{remark}

\newtheorem{remk}[thm]{Remark}

\newcommand{\id}{{\rm id}}
\newcommand{\dis}{\displaystyle}

\newcommand{\supp}{{\rm supp}}

\newcommand{\oL}{\overline{L}}

\newcommand{\Hcal}{ {\mathcal H}}

\newcommand{\Dcal}{ {\mathcal D}}

\newcommand{\Ecal}{ {\mathcal E}}
\newcommand{\Acal}{ {\mathcal A}}

\newcommand{\Ccal}{ {\mathcal C}}
\newcommand{\Bcal}{ {\mathcal B}}

\newcommand{\Hbb}{ {\mathbb H}}
\newcommand{\Rbb}{ {\mathbb R}}

\newcommand{\Ebb}{ {\mathbb E}}

\newcommand{\tHc}{\widetilde{\Hcal}}
\newcommand{\tOmega}{\widetilde{\Omega}}

\newcommand{\eps}{\varepsilon}

\newcommand{\spec}{{\rm spec}\,}
\newcommand{\dom}{{\rm dom}\,}

\newcommand{\alp}{\alpha}

\newcommand{\R}{{\mathbf R}}
\newcommand{\Z}{{\mathbf Z}}
\newcommand{\N}{{\mathbf N}}

\newcommand{\py}{\partial_y}
\newcommand{\px}{\partial_x}

\newcommand{\und}{\frac{1}{2}}

\newcommand{\td}{\frac{3}{2}}

\newcommand{\un}{\mathbbm{1}}    
\newcommand{\Cinf}{\mathcal{C}^\infty}
\newcommand{\Cinfz}{\mathcal{C}^\infty_0}

\newcommand{\tendsto}[2]{\xrightarrow[{#1} \rightarrow {#2}]{}}%
\renewcommand{\leq}{\leqslant}
\renewcommand{\geq}{\geqslant}

\newcommand{\display}{\displaystyle}

\definecolor{blue(ncs)}{rgb}{0.0, 0.53, 0.74}

\newcommand{\op}{\mathopen]}
\newcommand{\cl}{\mathclose[}

\newcommand{\pr}{\partial_r}
\newcommand{\ptheta}{\partial_\theta}

\title{Generic simplicity of ellipses}

\author{Luc Hillairet}
\address{Institut Denis Poisson, UMR 7013, CNRS-UO-UT, 
Université d'Orl\'eans, B\^atiment de Math\'ematiques, rue de Chartres, 
F-45100 Orl\'eans}
\email{\href{mailto:luc.hillairet@univ-orleans.fr}{luc.hillairet@univ-orleans.fr}}
\thanks{The work of L.H. is partially supported by the ANR grant ADYCT ANR-20-CE40-0017.}

\author{Chris Judge}
\address{Indiana University, Department of Mathematics, Bloomington, IN 47405} 
\email{\href{mailto:cjudge@indiana.edu}{cjudge@indiana.edu}}
\thanks{The work of C.J.\ is partially supported by a Simons 
Foundation MP-SCMPS-00006686. }

\begin{document}
\maketitle

\textit{Steve Zelditch's impressive work in a wide variety of subjects has been
very influential on ours. We cannot acknowledge all of the discussions with
him that eventually led to something, whether a result, a proof, an
example... Our use of his note \cite{Zelditch-Note} 
in this paper is representative of his influence on us. We hope that he would have liked the result and this text is dedicated to his memory.}

\section{Introduction}

In this article, we adapt the methods of \cite{HJCMP} to prove that 
the Laplace spectrum of the generic ellipse is simple, both with Dirichlet or Neumann boundary condition. Multiplicities are unchanged if one applies an isometry or a homothety, 
and hence it suffices to consider ellipses of the form
$$
\Ecal_{h}\,
:=\,
\{ (x,y)\, :\, (h\, x)^2 + y^2  < 1 \}.
$$
where $h \in \op 0, 1]$. 
\begin{thm}
\label{thm:main}
There exists a countable subset $\Ccal \subset \op 0,1]$ 
so that if $h \notin \Ccal$, then each eigenspace of the 
Dirichlet (resp.  Neumann) Laplace operator of the ellipse $\Ecal_h$ is 
one dimensional. 
\end{thm}

The spectral geometry of ellipses has a long history, starting with the work of \'Emile Mathieu 
(see \cite{Mth}) and it has been recently the subject of a number of striking results, both on its dynamical and spectral 
sides. In \cite{HZb}, H. Hezari and S. Zelditch, use (among many other ingredients) dynamical results initiated in \cite{AdSK} 
to prove that ellipses of small eccentricty are spectrally determined. In a previous paper (\cite{HZa}) the same authors conjecture that 
eigenvalues of a generic ellipse have multiplicity at most $2$. The latter result is a consequence of analyticity of the spectrum and the 
well-known fact that eigenvalues of the disk have multiplicity at most $2$. In this paper, 
we refine this result and prove that the spectrum of a generic ellipse is actually simple. 
It turns out that the method of proof also gives that there do exist ellipses (besides disks) that have at least one 
multiple eigenvalue. 

In \cite{HJCMP}, we designed a general approach to prove generic simplicity in settings that depend only on 
a finite number of geometric parameters. This should be constrasted with the well-known results of Albert 
\cite{Albert} and Uhlenbeck \cite{Uhlenbeck} that require infinitely many parameters. This approach relies on two main ingredients: 
analytic perturbation theory and (semiclassical) concentration of eigenfunctions. In \cite{HJCMP}, we used this approach 
for triangles and here we use it for ellipses. Although the general philosophy is the same, 
the exact method described in \cite{HJCMP} does not 
apply directly to ellipses. For the ellipse $\Ecal_h$ above,
the natural profile curve is $y=L(x)= \sqrt{1- (hx)^2}$ and the derivative 
$L'$ does not exist at the endpoints $x=\pm 1/h$. The method of \cite{HJCMP} 
requires that $L'$ be finite at each endpoint. To apply our approach in the case of ellipses, we 
need an extra new ingredient that is a careful control of the mass of eigenfunctions near $x = \pm 1/h$.
We also use the structure of the spectrum of the unit disk while the method of \cite{HJCMP} does not require a knowledge of the spectrum 
of a particular member of the family of domains.

We outline the contents of the paper. 
{We treat the case of the Dirichlet Laplacian in the main part of the text and provide 
details of the Neumann case in Appendix \ref{sec:neu}.}
In \S \ref{sec:line}, we state and prove a non-concentration estimate for 
a 1-dimensional semi-classical family $P_h u= -h^2\,u'' + Vu$ of 
Schr\"odinger operators with an unbounded,
single-well potential $V$. In \S \ref{sec:interval}, 
we apply this non-concentration
estimate to prove an estimate for a family of semi-classical
Sturm-Liouville operators $A_h u= -h^2 L^{-1}( L u')' + L^{-2} u$ on  a finite interval where $L$ is a non-negative function that vanishes on the endpoints
and has exactly one critical point. 
These operators arise from na\"ively applying separation of variables 
to a domain $\Omega$ 
that lies between the $x$-axis and the graph of the function $L$.
The estimate is a non-concentration estimate for $H^1$ functions 
supported away from the endpoints.

In \S \ref{sec:sum}, we make sense of the  
formal sum $\Acal_h = \sum_k -h^2 L^{-1}( L u')' + (\pi k/L)^2 u$,
and we obtain a non-concentration estimate in this context. 
In section \S \ref{sec:shrink} we show that the quadratic form $a_h$
associated to $\Acal_h$ is comparable and in fact `asymptotic'  to 
the quadratic form $q_h(u) = \int_{\Omega} h^2 \cdot |u_x|^2 + |u_y|^2$
on $H^1_0(\Omega)$.  This allows us to use the non-concentration estimate
for $\Acal_h$ to identify the semi-classical limits of real-analytic eigenvalue 
branches $h \mapsto E_h$ of $q_h$ under Dirichlet and certain mixed conditions.

The quadratic form $q_h$ is the pull-back of the Dirichlet energy 
on a domain $\Omega_h$ obtained by stretching $\Omega$ horizontally.
In \S \ref{sec:symmetric} we consider the domain $\tOmega_h$ obtained
by reflecting $\Omega$ across the $x$-axis. For example, $\tOmega_h$ is 
the ellipse $\Ecal_h$ if $L(x)= \sqrt{1-x^2}$. The symmetry of $\tOmega_h$ 
allows one to decompose into `even' and `odd' functions.  
The results of \S \ref{sec:shrink} immediately imply that 
limits of `odd' eigenvalue branches and the limits of `even' 
eigenvalue branches are distinct. By specializing this result 
to the family of ellipses  and by using Bourget's hypothesis 
we are able to  prove Theorem \ref{thm:main}.

In \S \ref{sec:ellipsoids} we use the same method to prove
that the generic ellipsoid in $\R^3$ has simple spectrum.

\newpage

\section{A one-dimensional non-concentration estimate}
\label{sec:line}

For each $u\in \Cinf_0(\R)$ and $h \in \op 0, \infty \cl$, define 
$
  P_hu\,=\,-h^2 \cdot u''\,+\,V \cdot u.
$
That is, $P_h$ is a one-dimensional semi-classical Schr\"odinger operator.
For basic results on Schr\"odinger operators, see for example \S 7.1 of  \cite{Borthwick}.

We will assume that the `potential' $V$ satisfies the following conditions:
\begin{enumerate}
\item[V1.] $V$ is positive and smooth and $\lim\limits_{\pm \infty} |V(x)|\,=\,+\infty.$

\item[V2.] If  $x\neq 0$, then $x \cdot V'(x)>0$.

\end{enumerate}

Condition V2 implies that $V$ attains its 
global minimum at $x=0$ and hence $P_h$ is semi-bounded from below. We will use $P_h$ to denote the Friedrichs extension.\footnote{In fact, $P_h$ is essentially self-adjoint, but this is not important for us.} 
Condition V1 implies that $P_h$ has compact resolvent, and so its spectrum,  $\spec(P_h)$, is discrete and consists of eigenvalues.
Condition V1 also implies that each eigenspace of $P_h$ is 1-dimensional. 
For each eigenvalue $\lambda$, choose an $L^2$-normalized
eigenfunction $\psi_{\lambda}$.
Then $\dis (\psi_{\lambda})_{\lambda\in \spec(P_h)}$ is a Hilbertian basis for $L^2(\R)$: Each $u \in L^2(\R)$  may be written as
\[
u\,
=\,
\sum_{\lambda \in \spec(P_h)} \langle 
\psi_\lambda, u\rangle 
\cdot 
\psi_\lambda
\]
where
$
\dis
\langle v,u \rangle = \int_{\Rbb} \bar{v} \cdot u\, dx
$.

It is customary to set $\dis   H^1_{P_h} := \dom ((P_h+1)^{\frac{1}{2}})$, and 
\[
  \begin{split}
    \| u\|_{H^1_{P_h}}^2& := \| (P_h+1)^{\frac{1}{2}}u\|_{L^2}^2\\
&= 
h^2\int_\R |u'(x)|^2 \,dx\,
+
\,\int_{\R} V(x) \cdot |u(x)|^2\,dx\,
+
\,
\int_{\R}|u(x)|^2\, dx.\\
    & = \sum_{\lambda\in \spec(P_h)} (\lambda\,+\,1) 
    \cdot 
    |\langle \psi_\lambda, u\rangle|^2.
  \end{split}
\]

We denote by $H^{-1}_{P_h}$ the dual space to $H^1_{P_h}$ equipped with the dual norm.
If $u\in H^1_{P_h}$ and if  $E\in \R$, then $(P_h-E)u$ belongs to $H^{-1}_{P_h}$.
We have 
\begin{eqnarray}
    {\|(P_h-E)u\|_{H^{-1}_{P_h}}}
    &=&
    \left(
    \sum_{\lambda\in \spec P_h} \frac{(\lambda-E)^2}{\lambda+1}
    \cdot 
    |\langle 
    \psi_\lambda,u\rangle|^2
    \right)^{\und}
    \label{eqn:H-minus-1}
    \\
  &=&
    \sup    
    \left\{
    \frac{\big|h^2\int \bar{u}'(x)v'(x) \,dx \,+\,\int (V(x)-E)\bar{u}(x)v(x)\,dx \big|}{\|v\|_{H^1_{P_h}}}, v\in H^1_{P_h} 
    \right\}
    \nonumber
\end{eqnarray}

The following proposition implies a non-concentration result for $O(h)$ quasimodes.

\begin{prop}
\label{prop:ncline}
  Let $K$ be compact set such that 
  $K\subset \mathopen] V(0),+\infty \mathclose[$.\\
  For any $\eps>0$ there exists a
  constant $C$ and $h_0$ such that, if $h\in \mathopen]0 ,h_0 \mathclose]$,  $E\in K$
  and $u\in \Cinf_0(\R)$, then 
  \[
    \|u\|_{H^1_{P_h}}\,
    \leq\,
    \frac{\eps}{h} \cdot \|(P_h-E)u\|_{H^{-1}_{P_h}}\,
    +\,
    C
    \cdot 
    \|h \cdot u'\|_{L^2}.
  \]
\end{prop}

\begin{proof}
  By assumption, there exists an interval $[a,b]$ such that $K\subset \op a,b \cl$ and $[a,b] \subset \op V(0),+\infty \cl$. We argue by contradiction.
  Suppose that the statement is not true, and let $\N^*$ denote the
  set of positive integers. 
  Then there exists $E_0\in K$, a constant $M$, 
  a subset $\Hbb \subset \op 0, \infty \cl$ with zero as an  accumulation point, and
  functions $E: \Hbb \to \R$ and $\eps: \Hbb \to \R$, such that as 
  $h \in \Hbb$ tends to zero, we have  $E_{h} \to E_0$, $\eps_{h} \to 0$, 
  \begin{equation}
  \label{eqn:quasimode}
  \| (P_h-E_h)u_h\|_{H^{-1}_{P_h}}\,
  \leq\,
  M 
  \cdot 
  h
  \cdot
  \|u_h\|_{H^1_{P_h}},
  \end{equation}
  and 
  \begin{equation}
  \| h \cdot u'_h\|_{L^2}\,
  \leq\, 
  \eps_h
  \cdot 
  \|u_h\|_{H^1_{P_h}}.
\end{equation}
  In the language of semiclassical analysis, 
  equation (\ref{eqn:quasimode})
  means that $u_h$ is a quasimode for $P_h$ of order $O(h)$.
  
The assumptions on $V$ imply that the compact interval $[a,b]$ consists of non-critical energies.
  Using semiclassical analysis, the distribution of the eigenvalues of $P_h$ in $K$ is thus well-understood.\footnote{See 
  \cite{Dimassi-Sjostrand}, \cite{Zworski}, \cite{HMR}, \cite{CdV}
  and the references therein.}
  In particular, using Bohr-Sommerfeld rules for instance, we know that the spectrum 
  in the interval $[a,b]$ is given by an (ordered) sequence $\lambda_{i,h}$ and that there exist two positive constants $c_1$ and $c_2$ such that 
  if  $h$ is small enough, then, for any eigenvalue $\lambda_{i,h}$ of $P_h$ in $[a,b]$ we have 
\begin{equation}
\label{eqn:dist-est}
c_1 \cdot h~
    \leq~
    \lambda_{i+1,h}-\lambda_{i,h}~
    \leq~
    c_2 \cdot h.
\end{equation}
The number $E_h$ lies in the closure of some component of $\R \setminus \spec P_h$, thus, there exists a unique integer $i_h$ so that
\[ 
\lambda_{i_h,h}~
\leq~ 
E_h~
<~
\lambda_{i_h+1,h}~
\leq~
\lambda_{i_{h},h}\,+\,c_2 \cdot h.  
  \]
  It follows that, for each sufficiently large  integer
  $N$ and $h$ small enough:
  \[
    a~
    <~
    \lambda_{i_h-N,h}~
    \leq~
    \lambda_{i_h,h}-N \cdot c_1
    \cdot h~
    \leq~
    \lambda_{i_h,h}~
    \leq~
    \lambda_{i_h,h}~
    +~
    N
    \cdot 
    c_1
    \cdot 
    h~
    \leq~
    \lambda_{i_h+N,h}~
    <~
    b. 
  \]
Using $N$, we decompose $u_h$  onto those modes that are close to $E_h$ and those that are far from $E_h$.\\
 Concretely, we define the 
 projection of  $u_h$ onto the modes close to $E_h$ to be 
 {
  \[
    v_h\,
    :=\, 
    \sum_{j=-N}^N \langle \psi_{{i_h+j,h}}, u_h \rangle 
    \cdot 
    \psi_{i_h+j,h}. 
  \]
  }
  We now show that $v_h$ approximates $u_h$ as $h$ goes to zero.
  Define $r_h:= u_h-v_h$ to be the remainder.
  Using the definitions of the $H^1_{P_h}$ and $H^{-1}_{P_h}$ norms we have
\begin{equation}
\label{eqn:remainder}
    \| r_h\|^2_{H^1_{P_h}}\,\leq\,
    \max 
    \left\{
    \frac{(\lambda_{i_h-N-1,h}+1)^2}{(\lambda_{i_h-N-1,h}-E_h)^2},\frac{(\lambda_{i_h+N+1,h}+1)^2}{(\lambda_{i_h+N+1,h}-E_h)^2}
    \right\}
    \cdot 
\| (P_h-E_h)u_h\|_{H^{-1}_{P_h}}^2. 
\end{equation}
The prefactor on the right of (\ref{eqn:remainder}) is bounded above by 
$
  \display \frac{(b+1)^2}{(Nc_1-c_2)^2 \cdot h^2}\,
  \leq   
    \frac{c_3}{N^2\cdot h^2}
$
for {$N$} large.
By combining this with (\ref{eqn:quasimode})---the assumption that $u_h$ 
is a quasimode---we obtain the following estimate for $h$ small and $N$ large:
  \[
    \| r_h\|^2_{H^1_{P_h}}\,
    \leq\,
    \frac{c_3}{N^2} \cdot  \| u_h\|_{H^1_{P_h}}^2.
  \]
  Here $c_3$ is a constant that depends neither on $h$ nor on $N$.

We next proceed to estimate $\| v_h\|^2$ by comparing it to 
$\|h \cdot v_h'\|_{L^2}^2$.
Set $V_{h,N}$ to be the span of 
$
\dis \{ \psi_{i,h}\, :\, |i- i_h| < N \}.
$
By construction $v_h$ belongs to each $V_{h,N}$.
The space $V_{h,N}$ is a $2N+1$-dimensional vector space.

  Let $\Bcal_{h,N}: V_{h,N} \to \R$ denote the quadratic form
  defined by 
  $
  \Bcal_{h,N}(v)\,
  =\,\| 
  h \cdot v'\|_{L^2}^2
  $.
This quadratic form is represented by a hermitian matrix $B_{h,N}$ in the 
basis of eigenfunctions.

The small $h$ behavior of the diagonal entries of $B_{h,N}$
is described by  semiclassical  {(defect)}  measures
{(see for example \cite{Zworski})}.
In  dimension $1$, there is only one  semiclassical measure
associated to a noncritical energy.
Indeed, this measure, denoted $\mu_{E_0}$, equals 
the Liouville measure on $T^* \R$ restricted
  to the energy shell\footnote{See the appendix 
  of \cite{HMR} for precise statements.}
  $$
  \Sigma_{E_0}~
  =~
  \{ (x,\xi)\, :\, \xi^2\, +\, V(x)\, =\, E_0 \}.
  $$
From this, we find that each  diagonal entry of $B_h$ converges to
the positive number 
 \[
  \beta\,
  =\,
  \int_{\Sigma_{E_0}} \xi^2 \,d\mu_{E_0}.
\]
A nice argument of Steve Zelditch in \cite{Zelditch-Note} when semiclassically reformulated, 
shows that, for any $j,k$ between $-N$ and $N$, there exists a (complex) measure $\mu_{jk}$ such that 
\[
\forall a,~~\langle \mathrm{Op}_h(a)\psi_{i_h+j},\psi_{i_h+k}\rangle 
\tendsto{h}{0} \int a d\mu_{jk},
\]
and that, moreover, the measure $\mu_{jk}$ is absolutely continuous with respect to 
$\mu_{E_0}$. We can thus write 
\[
\mu_{jk} = f_{jk}\,\mu_{E_0}
\]
for some integrable $f_{jk}$.
It follows that the matrix $\Bcal_{h,N}$ converge to the matrix 
\[
B_{0} = 
\left[
\int_{\Sigma_{E_0}} \xi^2 f_{jk} d\mu_{E_0}
\right]_{\substack{-N\leq j\leq N \\-N \leq k \leq N}}
\]
By construction, this matrix is Hermitian and we claim that it is also positive. For, otherwise, there would be a normalized eigenvector 
$(w_{j})_{-N\leq j\leq N}$ such that 
\[
\int_{\Sigma_{E_0}} \xi^2 \sum_{j,k} \overline{w_j}f_{jk}w_k d\mu_{E_0} \leq 0.
\]
The latter implies that the integrable function $\sum\limits_{j,k} \overline{w_j}f_{jk}w_k$ vanishes. This is in contradiction with 
the fact that the measure $\sum\limits_{j,k} \overline{w_j}f_{jk}w_k d\mu_{E_0}$ is the semiclassical measure associated with the sequence 
\[
w_h = \sum_{j=-N}^N w_j \psi_{i_h+j,h},
\]
and, as such, is a probability measure. 

Denote by $\beta$ the smallest (positive) eigenvalue of $B_0$.  

For $h$ small enough, we have
  \[
    \begin{split}
      \| h \cdot v_h'\|^2 & \geq\, \frac{\beta}{2} \cdot \|v_h\|_{L^2}^2\\
      & \geq \frac{\beta}{2(b+1)} \cdot \|v_h\|^2_{H^1_{P_h}} \\
      & \geq \frac{\beta}{2(b+1)} \cdot \left(\|u_h\|^2_{H^1_{P_h}}-\|r_h\|^2_{H^1_{P_h}} \right)\\
      & \geq \frac{\beta}{2(b+1)} \cdot 
      \left(1-\frac{c_3}{N^2} \right)
      \cdot \|u_h\|^2_{H^1_{P_h}}.
    \end{split}
  \]
  This yields a contradiction with the definition of the sequence $\eps_h$ 
  provided that  $N$ is sufficiently large.
\end{proof}

\begin{remk} 
Note that, by using a semiclassical quantization, the controlling term 
$\| h \cdot v_h'\|_{L^2}^2$ may be thought of 
as $\langle \mathrm{Op}_h(\xi^2) u_h, u_h \rangle$. 
In the proof we can replace $\xi^2$ by any non-negative symbol $a$ of order $0$ provided that $\dis \int a\, d\mu_{E_0} >0$.
\end{remk}

\begin{remk}
Because $\int \xi^2\, d\mu_{E_0}>0$,
the probability measure 
$\mu_{E_0}$ is not concentrated at the turning points for noncritical $E$.
This can also be observed by using an Airy function analysis near a turning point (see \cite{HJCMP} Proposition 9.1 and Remark 9.3).
Proposition \ref{prop:ncline}
can thus be rephrased as \textit{any $O(h)$ quasimode cannot concentrate on the turning points}. The corresponding statement for an $o(h)$ quasimode
is a straightforward implication of the invariance of semiclassical measures under the hamiltonian flow of the symbol. 
\end{remk}


\newpage

\section{An application to a singular Schrödinger operator}
\label{sec:interval}

To study the eigenvalue problem for an ellipse, we will use
singular Schrödinger operators on $\op -1, 1 \cl$.
In this section we derive a non-concentration estimate of the same
form as above.

Let $L$ be a smooth positive function on $\op -1, 1 \cl$, that satisfies the following assumptions:
\begin{enumerate}
\item[L1.] For any $x\in \op -1, 1 \cl$, we have $x \cdot L'(x) < 0$ when $x\neq 0$.
\item[L2.]
$L$ extends continuously to $[-1,1]$ by setting $L(\pm 1)=0.$
%
\end{enumerate}
%
The function $L^{-2}$ will be used to construct a potential 
that satisfies conditions V1 and V2 of \S \ref{sec:line}.

\begin{remk}
    For our application to the disk, we will have that $L$ vanishes like $\sqrt{(1-x)}$ at $x=1$ (resp. $\sqrt{1+x}$ at $x=1$). We note here 
    that for functions $L$ that either do not vanish at $x=\pm 1$ or that vanish to order $1$, then the results below would follow directly of \cite{HJCMP}.
\end{remk}

Let $\Hcal = L^2 \left(\op -1, 1 \cl, L(x)\,dx \right)$ 
and  let $A_h$ be the symmetric operator defined on $\Cinf_0(\op -1, 1 \cl)$ by 
\[
A_hu\,
:=\, 
-h^2
\cdot 
\frac{1}{L}(Lu')'\,+\,\frac{1}{L^2}u.
\]
The study of the essential self-adjointness of $A_h$ is a standard procedure that we do not pursue here.
We simply observe that the quadratic form $\langle A_h u,u\rangle_{\Hcal}$ is non-negative, so that we may consider 
the Friedrichs extension that we still denote by $A_h$. 

The operator $A_h$ is non-negative and has compact resolvent so that its spectrum consists of eigenvalues and we may define spaces 
$H^1_{A_h}$ and $H^{-1}_{A_h}$ in the same way that 
we defined $H^1_{P_h}$ and $H^{-1}_{P_h}$
as in the previous section.  

For  $\delta\in \op 0,1 \cl$, we denote by $I_\delta$ the interval $[-1+\delta, 1-\delta]$.
The following proposition is a consequence of proposition \ref{prop:ncline}.

\begin{prop}
\label{prop:ncinterval}
\ 
  Let $b> a > 1/L(0)^2$.
  There exists $\delta_0\in \op 0,1 \cl$ such that the following holds. 
  For any $\delta\in \op 0,\delta_0 \cl$ and any $\eps>0$ there exist
  constants $C$ and $h_0$ such that if  $h\in \op0,h_0 \cl$, if  $E\in [a,b]$
  and if  $u\in \Cinf_0(I_\delta)$ then 
  \[
    \|u\|_{H^1_{A_h}}\,\leq\,\frac{\eps}{h}
    \cdot 
    \|(A_h-E)u\|_{H^{-1}_{A_h}}\,
    +\,
    C \cdot \|h \cdot u'\|_{\Hcal}.
  \]
\end{prop}
\begin{remk}
  We warn the reader that the interval $[a,b]$ in the statement of  
  Proposition \ref{prop:ncinterval} 
  corresponds to the compact set $K$ in the statement Proposition \ref{prop:ncline}
  and not to the interval used in the proof of Proposition \ref{prop:ncline}. 
\end{remk}

\begin{proof}
Because $L$ is continuous and $L(-1)=0=L(1)$, there exists $\delta_0 \in \op 0, 1\cl$ 
so that $x \notin I_{\delta_0}$ implies that $L(x)^{-2} \geq b+1$.
    Let $\delta \in \op 0, \delta_0\cl$.
  Choose a potential $V_\delta$ that coincides
  with $\frac{1}{L^2}$ on $I_{\frac{\delta}{2}}$ 
  and that satisfies assumptions V1 and V2 of \S \ref{sec:line}. 
  See Figure \ref{fig:potential}.
  Define $P_h u:= -h^2 \cdot u'' + V_{\delta} \cdot u$.

\begin{figure}
\label{fig:potential}
\caption{The construction of the potential $V_{\delta}$.}
\includegraphics[scale=.3]{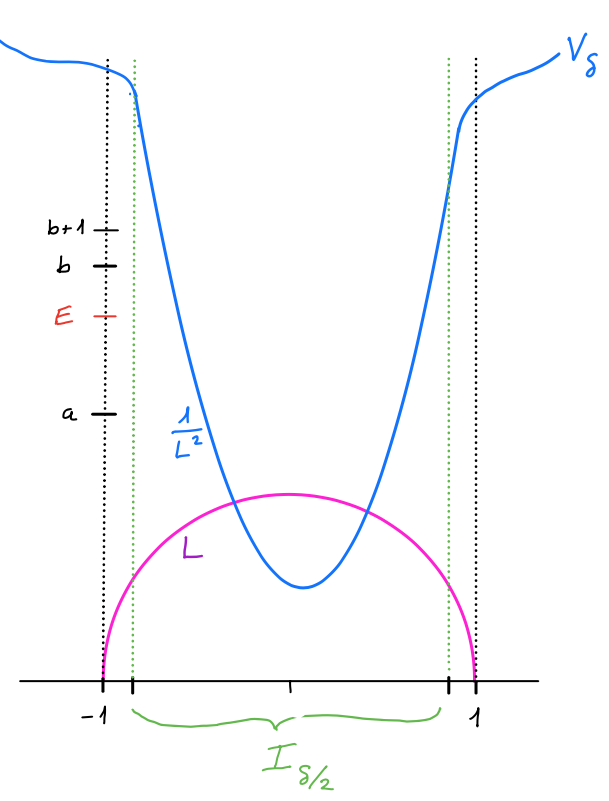}
\end{figure}
  
  Let $\epsilon>0$.
  Proposition \ref{prop:ncline} applies to give a constant $C$ so that 
  for any $u\in \Cinf_0(I_\delta)$
  $$
    \|u\|_{H^1_{P_h}}\,
    \leq\,
    \frac{\eps}{h} 
    \cdot 
    \|(P_h-E)u\|_{H^{-1}_{P_h}}\,
    +\,
    C 
    \cdot 
    \|h \cdot u'\|_{L^2(\R)}.
  $$
 The claim now follows from \ref{lem:equivAP} below. 
\end{proof}

\begin{lem}\label{lem:equivAP}
Let $a,b,\delta$ and $V_\delta$ be as 
in the proof of Proposition \ref{prop:ncinterval}.
Then, there exist constants $m_i$ and $M_i,~i=1,2,3$, such that
for any $u\in \Cinf_0(I_\delta)$ 
\begin{eqnarray}
m_1
\cdot 
\| u\|_{L^2(\R)}
&\leq &
\|u\|_{\Hcal}~
\leq~
\frac{1}{m_1}
\cdot 
\|u\|_{L^2(\R)} 
\label{eqn:norms-equiv}
\\
m_2
\cdot 
\|u\|_{H^1_{P_h}}
&\leq & 
\|u\|_{H^1_{A_h}}~
\leq~
\frac{1}{m_2}
\cdot 
\|u\|_{H^1_{P_h}}
\label{eqn:H1norms-equiv}
\\
\|(P_h-E)u\|_{H^{-1}_{P_h}}
&\leq &
\frac{1}{m_3} \cdot \|(A_h-E)u\|_{H^{-1}_{A_h}}~
+~
m_3 \cdot h \cdot \|h \cdot u'\|_\Hcal.
\label{eqn:quasi-equiv}
\end{eqnarray}
\end{lem}

\begin{proof}
Everywhere in the proof, we will use the fact that $u$ has support in ${I_\delta=[-1+\delta,1+\delta]}$.

Estimate (\ref{eqn:norms-equiv}) follows from the fact that $L$ is uniformly bounded above and below on $I_\delta$.
For (\ref{eqn:H1norms-equiv}) we observe that 
\begin{eqnarray*}
\|u\|_{H^1_{P_h}}^2 
&=&
\int_{I_\delta} |h \cdot u'(x)|^2 \,dx~
+\,\int_{I_\delta} (V(x)+1) \cdot |u(x)|^2\, dx 
\label{eqn:P_h-square}
\\
\|u\|_{H^1_{A_h}}^2 
&=& 
\int_{I_\delta} | h \cdot u'(x)|^2 \,L(x)\, dx~
+~
\int_{I_\delta} 
\left(
\frac{1}{L(x)^2}+1
\right) 
\cdot 
|u(x)|^2\,L(x)\, 
dx.
\label{eqn:A_h-square}
\end{eqnarray*}

and we obtain (\ref{eqn:H1norms-equiv}) using (\ref{eqn:norms-equiv}) and
the fact that the functions $V$ and $\frac{1}{L^2}$ coincide on $I_\delta$.

To prove estimate (\ref{eqn:quasi-equiv}), we use the definitions of the $H^{-1}$ norms.
In particular, if we let 
\[
I(\phi)\,
:=\,
h^2 \int_\R \bar{u}'(x)\, \phi'(x)\,dx\,
+\,
\int_\R (V(x)-E)\, \bar{u}(x)\, \phi(x)\,dx.
\]
then  $\|(P_h - E)u\|_{H^{-1}_{P_h}}$ equals the supremum of $|I(\phi)|/\|\phi\|_{H^1_{P_h}}$ 
over  $\phi \in \Cinf_0(\R)$.
We now proceed to compute $I(\phi)$.
Because $u$ has support in $I_\delta$, we can insert a cutoff function 
$\chi$ that is $1$ on $I_\delta$ and $0$ outside $I_{\delta/2}$.
Since $L>0$ on $I_{\delta/2}$, there exists $\psi \in \Cinf_0(\R)$ 
such that $L \cdot \psi=\phi \cdot \chi$. 
We have
\begin{eqnarray*}
I(\phi) 
&=&
h^2 \int_\R \bar{u}'(x) \, (\chi \cdot \phi)'(x)\,dx\,
+\,
\int_\R (V(x)-E)\, \bar{u}(x)\, (\chi \cdot \phi)(x)\,dx
\\
&=&
h^2 \int_\R\bar{u}'(x) \, (\psi \cdot L)'(x)\,dx\,
+\,
\int_\R (V(x)-E)\, \bar{u}(x)\, (\psi \cdot L)(x)\,dx
\\
&=& 
h^2 \int_\R \bar{u}'(x) \, \psi'(x)\,L(x) \,
dx\,
+\,
\int_\R 
\left(
\frac{1}{L(x)^2} - E
\right)\,
\bar{u}(x)\,
\psi(x)\,
L(x)\,
dx\,
+\,
h^2 \int \bar{u}'(x)\,
\psi(x)\,
L'(x)\,
dx.
\\
&=& 
\left\langle (A_h -E)u, \psi \right\rangle_{\Hcal}~
+~
h^2 \int \bar{u}'(x)\,
\psi(x)\,
L'(x)\,
dx.
\end{eqnarray*}
Using the definition of the $H^{-1}_{A_h}$ norm and the fact that $L'$ is bounded on $I_\delta$,  we obtain
\[
  |I(\phi)|~
  \leq~
  \|(A_h-E)u\|_{H^{-1}_{A_h}}
  \cdot 
  \|\psi\|_{H^1_{A_h}}\,+\, c \cdot h \cdot \|hu'\|_{L^2(\R)}
  \cdot 
  \|\psi\|_{L^2(\R)}.
\]
Since $\psi$ has support in $I_{\delta/2}$ we have 
the estimate $\|\psi\|_{H^1_{A_h}} \,\leq\,C\, \| \psi\|_{H^1_{P_h}}$.\\
Since  $\psi\,=\,\frac{\chi \cdot \phi}{L}$ it is straightforward to prove that
\[
  \| \psi\|_{H^1_{P_h}}\,\leq\,C\, \| \phi\|_{H^1_{P_h}}.
\]
Finally, we obtain
\[
  |I(\phi)|\,
  \leq\, 
  \left( C \cdot \|(A-E)u\|_{H^{-1}_{A_h}}\,
  +\,
  c\, h\,  \|h \cdot u'\|_{L^2(\R)}\right)\| \phi\|_{H^1_{P_h}}. 
\]
This yields (\ref{eqn:quasi-equiv}).
\end{proof}

\newpage


\section{Summing non-concentration estimates}
\label{sec:sum}

Let $L$ satisfy conditions L1 and L2 of the previous section and as before 
let $\Hcal = L^2 \left(\op -1, 1 \cl, L(x)\,dx \right)$.
For each $k \in \N^*$, $h \in \op 0, \infty \cl$, and $v \in \Cinfz(\op -1,1\cl)$ define 
$$
A_{k,h} v~
=~
-h^2\,
  \frac{1}{L} \left(L v' \right)'~
  +~
  \frac{k^2\pi^2}{L^2(x)}\, v.
$$
We will let $A_{k,h}: D_{k,h} \to \Hcal$ denote the Friedrichs extension.
Note that the domain $D_{k,h}$ does not depend on $k$ or $h$ and hence we will simply write $D$. The analysis of section \ref{sec:interval} applies to each of the operators $A_{k,h}$. We will consider a `sum' of the $A_{k,h}$ over $k \in \N^*$.

Define the Hilbert space $\widetilde{\Hcal}\, :=\,\ell^2(\N^* ; \Hcal)$.
An element in $\widetilde{\Hcal}$ can be uniquely written as
\[
  u\,=\,\sum_{k\geq 1} u_k\otimes e_k
\]
where $(e_k)_{k\geq 1}$ is the canonical Hilbert basis of $\ell^2(\N^*)$.
Note that $\| u\|^2_{\tHc}\,=\,\sum_k \|u_k\|^2_{\Hcal}$,

Let $\widetilde{\Dcal}_f$ denote the subspace of $\widetilde{\Hcal}$
consisting of $u$ such that each $u_k \in D$ and only finitely many of the $u_k$ are nonzero.
(Here the subscript $f$ refers to the finiteness of each sum.)
Define the operator $\Acal_{f,h}: \widetilde{\Dcal}_f \to \widetilde{\Hcal}$ 
by
\begin{eqnarray*}
  \Acal_{f,h} u
  &=& 
  \sum_{k \geq 1} A_{k,h} u\, \otimes e_k
  \\
  &=&
  \sum_{k \geq 1} 
  \left(
  -h^2\,
  \frac{1}{L} \left(Lu_k' \right)'\,
  +\,
  \frac{k^2\pi^2}{L^2(x)}u_k
  \right)
  \otimes e_k
\end{eqnarray*}
The quadratic form $a_h$ associated with $\Acal_{f,h}$ is given by\footnote{A similar quadratic form was introduced in \S 11 of \cite{HJCMP}
as part of the so-called method of asymptotic separation of variables.
Here we include the otherwise unimportant factor of $\pi^2$ 
in the definition of $a_h$ in order to be consistent with the 
approach in \cite{HJCMP}.}
\begin{eqnarray}
\label{eqn:a-defn}
a_h(u)\,
&=&
\sum_{k\geq 1}\, h^2 \int_{-1}^1 |u_k'(x)|^2\,L(x)\,dx\,
+\,
\int_{-1}^1 \frac{k^2\pi^2}{L^2(x)}\, |u_k(x)|^2 \,L(x)\, dx
\\
&=&
\sum_{k \geq 1} a_{k,h}(u_k)
\nonumber
\end{eqnarray}

The following theorem summarizes the spectral theory of $\Acal_{f,h}$. 
Its proof is left to the reader.
Let $\widetilde{\Dcal}$ consist of $u \in \widetilde{\Hcal}$ such that
$u_k \in D$ for each $k$ and  $ \sum_k \|A_{k,h} u_k\|^2 \,<\,+\infty$.

\begin{thm}
The operator $\Acal_h: \widetilde{\Dcal} \to \widetilde{\Hcal}$ defined by 
\[
\Acal_h \left(\sum_{k\geq 1} u_k\otimes e_k\right)\,=\,\sum_{k\geq 1} A_{k,h}u_k\otimes e_k
\]
is the Friedrichs extension of $\Acal_{f,h}$. It has compact resolvent. 
Its spectrum is given by 
\[
\spec \Acal_h \,=\,\bigcup_{k\geq 1} \spec A_{k,h}.
\]
More precisely, if, for any $k$, $(\psi_{k,\ell})_{\ell\geq 0}$ is a Hilbert basis of $\Hcal$ consisting in 
eigenvectors for $A_{k,h}$ then $(\psi_{k,\ell}\otimes e_k)_{k,\ell \geq 0}$ is a Hilbert basis of eigenvectors of $\Acal_h$.
\end{thm}

The $H^1_{\Acal_h}$ norm is easily computed: 
\[
\| u\|^2_{H^1_{\Acal_h}}\,
=\,
\sum_{k\geq 1}\, \|u_k\|^2_{H^1_{A_{k,h}}},
\]
 and $u$ belongs to $H^1_{\Acal_h}$ if and only if the sum on the right hand side
is finite. 
By duality
\[
\| (\Acal_h -E) u\|^2_{H^{-1}_{\Acal_h}} \,=\,\sum_{k\geq 1} \|(A_{k,h}-E)u_k\|^2_{H^{-1}_{A_{k,h}}}.
\]

The analysis in \S \ref{sec:interval} leads to the following. 

\begin{prop}
\label{prop:ncsum}
Let $k_0 \in \N^*$ and suppose that 
\[
[a,b]~
\subset~
\biggr] 
\left(
\frac{k_0\,\pi}{L(0)}
\right)^2,\,
\left(
\frac{(k_0+1)\,\pi}{L(0)}
\right)^2
\biggr[.
\]
There exists $\delta_0\in \op 0,1 \cl$ such that, for any $\eps>0$, 
for any $\delta\in \op 0,\delta_0 \cl$, there exist
constants $C$ and $h_0$ such that, 
for any $h\in \op 0,h_0 \cl$, for any $E\in[a,b]$ and any 
$u\in H^1_{\Acal_h}\cap\, \ell^2(\N^*,\Cinf_0(I_\delta))$ 
the following estimate holds:
\[
\|u\|_{H^1_{\Acal_h}}\,
\leq\,
\frac{\eps}{h}\,
\| (\Acal_h -E)u\|_{H^{-1}_{\Acal_h}}\,+\,C\, \|h \cdot D' u\|_{\tHc},
\]
where $D'$ is the operator defined by $D'(\sum u_k\otimes e_k)\,=\,\sum u'_k\otimes e_k$. 
\end{prop}

\begin{proof}
Write $u= \sum_{k \geq 1} u_k\otimes e_k$.\\
For $k \in \N^*$, let $T_k$ denote the `threshold' $(k\pi/L(0))^2$.\\
Note that $a_{k,h}(v)\, \geq\, T_k\, \|v\|_{\Hcal}$ 
and hence $\spec A_{k,h} \subset [T_k,+\infty \cl$ for each $k$ and $h$.\\
In particular, if $k \geq k_0+1$ and $\lambda \in \spec A_{k,h}$, then by hypothesis
$$
\lambda - E~ \geq~ T_{k_0+1} -b\,  >\, 0.
$$
Thus, by using the analogue of equation (\ref{eqn:H-minus-1}), we find that
if $k \geq k_0+1$, then 
\[
\| u_k\|_{H^1_{A_k}}~
\leq~
c_+\,
\| (A_{k,h} -E)u\|_{H^{-1}_{A_k}},
\]
where 
\[
  c_+\,
  =\,
  \sup 
  \left\{ 
  \frac{|\lambda +1|}{|\lambda -E|}\, 
  :\,
  \lambda \geq T_{k_0+1} \mbox{ and } E \in [a,b]
  \right\}
  =\,
  1\, +\, \frac{b+1}{T_{k_0+1} -b}.
\]

For $k\leq k_0$, we may use the estimate in Proposition \ref{prop:ncinterval}:
\[
 \| u_k\|_{H^1_{A_k}}\,
 \leq\,
 \frac{\eps}{h}
\,
\| (A_{k,h} -E)u \|_{H^{-1}_{A_k}} \,
+\,
C\,
\|h \cdot u'_k\|_{\Hcal}.
\]
By squaring and summing all of these estimates, we obtain
\[
  \| u\|_{H^1_{\Acal_h}}^2\,
  \leq\,
  \left(\max
  \left\{\frac{2\eps}{h},\, 
  c_+
  \right\}\right)^2
  \cdot \| (\Acal-E)u\|_{H^{-1}_{\Acal_h}}^2\,
  +\,
  C\,
  \|h \cdot D'u\|_{\tHc}^2.
\]
We choose $h_0<\frac{2\eps}{c_+}$, and the claim follows.
\end{proof}

\newpage


\section{Compressing half-ovals with Dirichlet and mixed conditions}
\label{sec:shrink}

We apply the preceding estimates to study the analytic eigenvalue 
branches of the family $q_h$ of quadratic forms 
defined on certain subspaces of $H^1(\Omega)$ where
\begin{equation}
\label{eqn:half-oval}
\Omega \,=\,\left \{ (x,y)\in \R^2,~~|x|\,<\,1,~0\,<\,y\,<\,L(x)\right \}
\end{equation}
and  $L$ satifies the conditions L1 and L2 of section \ref{sec:interval}. 
For each $u\in H^1(\Omega)$, we define
\begin{equation}
\label{eqn:quad}
q_h(u)\,
:=\,
h^2 \int_{\Omega}\, |\partial_x u(x,y)|^2\,dx\, dy\,
+\,\int_{\Omega}\, |\partial_y u(x,y)|^2\,dx\, dy.
\end{equation}

We consider individually the restriction of $q_h$ to three 
distinct subspaces of $H^1(\Omega)$, 
and thus obtain three distinct spectral problems: 

\begin{itemize}

\item  {\it Full Dirichlet:} Restrict  $q_h$ to $H^1_0(\Omega)$.

\item {\it Dirichlet on curved part:} 
Restrict $q_h$ to the subspace $H^1_{0c}(\Omega)$  consisting of $u \in H^1(\Omega)$  whose trace vanishes on the graph of $L$. In this case,
eigenfunctions satisfy Dirichlet conditions on the graph of $L$ and 
Neumann conditions on the $x$-axis.

\item {\it Dirichlet on straight part:} We restrict $q_h$ 
to the subspace $H^1_{0s}(\Omega)$  consisting of $u \in H^1(\Omega)$
     that vanish on the $x$-axis.  In this case,
eigenfunctions satisfy Neumann conditions on the graph of $L$ and 
Dirichlet conditions on the $x$-axis.

\end{itemize}
Here we study the spectrum of $q_h$ relative to $\|\cdot\|^2$ where $\|\cdot \|$ 
denotes the standard norm 
on $L^2(\Omega, dx\, dy)$.
Recall that $u$ is an eigenfunction of $q_h$ relative to $\|u\|^2$ 
with eigenvalue $E$ if and only if $u$ belongs to the chosen subspace of $H^1(\Omega)$ and $q_h(u,v )= E\, \langle u, v\rangle$
for each $v$ belonging to the same subspace.

\begin{remk}
\label{remk:stretch}
The quadratic form $q_h$ arises from deforming the domain $\Omega$.
Indeed, suppose that $\Omega_h$ is the image of $\Omega$ under the map 
$(x,y) \mapsto (x/h, y)$. By making the change variables $\bar{x}=x/h$ we have 
\begin{equation*}
\label{eqn:stretch}
\frac{1}{h}\,
q_h(u)~
=~
\int_{\Omega_h}\, 
|\partial_{\bar{x}} u(\bar{x},y)|^2\, d\bar{x}\, dy\,
+\, 
\int_{\Omega_h}\, 
|\partial_{y} u(\bar{x},y)|^2\, d\bar{x}\, dy.
\end{equation*}
Note that the map $(x,y) \mapsto (x/h, y)$
`stretches' $\Omega$ in the horizontal direction by a factor of $1/h$. 

On the other hand, the map $(x,y) \mapsto (x, h\, y)$ compresses
the domain $\Omega$ in the vertical direction. If $\Omega_h$ is the image of 
this map and we set $\bar{y} = h \cdot y$, then we find that 
\begin{equation*}
\label{eqn:compress}
\frac{1}{h}\,
q_h(u)~
=~
\int_{\Omega_h}\, 
|\partial_{x} u(x,\bar{y})|^2\, dx\, d\bar{y}\,
+\, 
\int_{\Omega_h}\, 
|\partial_{\bar{y}} u(x,\bar{y})|^2\, dx\, d\bar{y}.
\end{equation*}
In either case, the study of the spectrum of Dirichlet and mixed
Laplace operators on $\Omega_h$
may be reduced to the study of the spectrum of $q_h$ restricted 
to the appropriate subpace of $H^1(\Omega)$ and relative to $\|\cdot\|^2$. 
Observe that the latter norm also is also simply related to the $L^2$ norm in $\Omega_h$ 
\end{remk}

Kato-Rellich analytic perturbation theory \cite{Kato} 
applies to the family $q_h$, and so in each of the three cases above, 
the spectrum of $q_h$ relative
to $\|\cdot\|^2$ may be organized into analytic eigenbranches $(E_h,u_h)$.
The derivative of the eigenvalue branch $E_h$ 
is given by the Feynman-Hellmann formula: 
\begin{equation}
\label{eqn-Feynman-Hellmann}
\frac{dE_h}{dh}\,
=\,
\frac{(\partial_h q_h)(u_h)}{\|u_h\|^2}~
=~
\frac{2 h}{\|u_h\|^2} \int_{\Omega} |\partial_x  u_h(x,y)|^2\, dx\, dy~
=~
\frac{1}{h}\, \frac{\| h \cdot \partial_x  u_h(x,y)\|^2}{\|u_h\|^2}.
\end{equation}
In particular, $\partial_h E_h \geq 0$ and so
each eigenvalue branch $E_h$ is an increasing function of $h$. 
Since $E_h \geq 0$ for all $h$, we deduce that $E_h$ tends to 
a limit as $h$ tends to $0$. 

The following theorem restricts the possible limits 
of eigenvalue branches in the full Dirichlet case.
After the proof we will state the analogous result for the 
mixed cases and we will discuss how to modify the proof.

\begin{thm}[Full Dirichlet]
\label{thm:limits}
For any analytic eigenvalue branch $E_h$ of $q_h$ on $H^1_0(\Omega)$ relative to 
$\|\cdot\|^2$, there exists $k\in \N^*$ such that 
\[
\lim_{h \to 0}\,  E_h~ =~ \left( \frac{k\, \pi}{L(0)} \right)^2.
\]
\end{thm}

We will use two lemmas  to prove Theorem \ref{thm:limits}.

The first lemma is a corollary of the Poincaré inequality on the segment $[0,L(x)]$.
For $\delta \in \op 0,1 \cl$ define 
$J_{\delta}:= \op -1, 1 \cl \setminus I_{\delta}$ 
where $I_{\delta}$ is defined as in \S \ref{sec:interval}.

\begin{lem}[Poincar\'e estimate]
\label{lem:Poinc}
There exists $f: \op 0,1 \cl \to \R^*$ with $\lim_{\delta \to 0} f(\delta)=0$
so that if $\delta\in  \op 0,1 \cl$ and $u\in H^1_0(\Omega)$, then  
\[
\int_{\Omega} \un_{J_{\delta}}(x) \cdot 
|u(x,y)|^2\, 
dx\, dy~
\leq~ 
f(\delta)
\int_{\Omega}\un_{J_\delta}(x)
\cdot |\partial_y u(x,y)|^2\,
dx\,
dy.
\]
\end{lem}

\begin{proof}
For each fixed $x \in \op -1, 1 \cl$
we perform a Fourier sine decomposition in  $y$:
\begin{equation}
\label{eqn:fourier}
u(x,y)\,
=\,
\sum_{k\geq 1} u_k(x) \cdot \sqrt{2}\, \sin \left(\frac{k\pi }{L(x)}\, y \right).
\end{equation}
A computation gives
\[
\begin{split}
\int_{\Omega} \un_{J_{\delta}}(x) \cdot 
|u(x,y)|^2\, 
dx\, dy~
&=~
\sum_{k\geq 1}\, \int_{J_\delta} |u_k(x)|^2\,
L(x)\,
dx
\\
\int_{\Omega}\un_{J_\delta}(x)
\cdot |\partial_y u(x,y)|^2\,
dx\,
dy~
&=~ 
\sum_{k\geq 1}\,
\int_{J_\delta} 
\left(
\frac{k\pi}{L(x)}
\right)^2 \,
|u_k(x)|^2\, 
L(x)\,
dx.
\end{split}
\]
Thus the desired inequality holds with
$$
f(\delta)~ :=~ \sup_{x \in J_{\delta}} \left(\frac{L(x)}{\pi} \right)^2.  
$$
Since $L(x) \to 0$ as $x \to \pm 1$, we have $\dis \lim_{\delta \to 0} f(\delta)=0$. 
\end{proof}

The second lemma will show that the quadratic form $q_h$ may 
be approximated by the quadratic form $a_h$ of section \ref{sec:sum}.\footnote{This approximation is the basis for the method of asymptotic separation of variables described in \cite{HJCMP}.} 
To be precise, we identify the space $L^2(\Omega, dx\, dy)$ 
with the space $\ell^{2}(\N^*, \Hcal)$ 
using the unitary map $u \mapsto \sum_k u_k  \otimes e_k$ where  
$u_k$  is  as in  equation (\ref{eqn:fourier}).
Under this identification, for $u \in \Cinf_0(\Omega)$ we set 
\begin{eqnarray*}
a_h(u)\,
&=& 
\sum_{k\geq 1}\,
h^2
\int_{-1}^1  
|u'_k(x)|^2\,L(x)\,
dx\,
+\,
\int_{-1}^1 
\left(\frac{k\pi}{L(x)}\right)^2\, |u_k(x)|^2\, L(x)\,
dx
\\
&=&
\|h \cdot D'u\|^2\,
+\, 
\|\py u\|^2
\end{eqnarray*}
where $D'$ is the operator defined in Proposition \ref{prop:ncsum}.
As in section \ref{sec:sum}, we will let $\Acal_h$ denote the 
self-adjoint operator associated to $a_h$.

For any $\delta \in \op 0,1 \cl$, we define
\[
\Omega_\delta\,
=\,
\big\{ (x,y)\in \Omega\, :\, x \in I_{\delta}\}.
\]
where we recall that $I_{\delta}= \op -1+\delta, 1-\delta \cl$.
On the domain $\Omega_{\delta}$, the function $L'/L$ is
uniformly bounded. 

\begin{lem}[Asymptotic at first order]
\label{lem:asysep}
For any $\delta \in \op 0,1 \cl$, there exists a constant $C_\delta$ such that, 
for any $u \in H^1_0(\Omega_\delta)$ and $v \in H^1_0(\Omega)$ we have 
\begin{equation}
\label{eqn:asysep}
\left| q_h(u,v)\, -\, a_h(u,v)\right|~
\leq~
C_\delta\, h\,  a_h(u)^{\frac{1}{2}}\, a_h(v)^{\frac{1}{2}}.
\end{equation}
\end{lem}

\begin{proof}
We first assume that $v \in  H^1_0(\Omega_{\delta})$ and then later 
show how to extend the inequality to $v \in  H^1_0(\Omega)$.

Since $L'/L$ is uniformly bounded on $I_{\delta}$, if $w \in  H^1_0(\Omega_\delta)$, then the function  $w^*$ defined by
{
$$
w^*(x,y)\,
:=\,
-y\, \frac{L'(x)}{L(x)} \cdot w(x,y)
$$
}
also belongs to $ H^1_0(\Omega_{\delta})$.
A computation shows that
\begin{equation}
\label{eqn:D'}
\px u\,
=\,
D' u + \py u^*.
\end{equation}
Using this, a further computation shows that
$$
q(u,v)-a(u,v)\,
=\,
h^2
\cdot
\left(\,
\langle D'u, \py v^* \rangle\,
+\,
\langle \py u^*,D'v \rangle\,
+\,
\langle \py u^*, \py v^* \rangle\,
\right)
$$
where $\langle \cdot, \cdot \rangle$ is the inner product for 
$L^2(\Omega, dx\, dy)$.
Let $C = \sup\{ |L'(x)/L(x)|\, :\,  x \in I_{\delta}$\}. 
By using the Cauchy-Schwarz inequality we find that 
$$
|q(u,v)-a(u,v)|\,
\leq\,
h^2
\cdot
\left(\,
C\, \|D'u\|\, \|\py v \|\,
+\,
C\,
\| \py u\|\, \|D'v \|\,
+\,
C^2\,
\|\py u\|\, \| \py v\|\,
\right).
$$
In general, $h\, \|D' w\| \leq a(w)^{\und}$ 
and $\|\py  w\| \leq a(w)^{\und}$.
Thus, we obtain (\ref{eqn:asysep}) when $v \in H^1_0(\Omega_{\delta})$.

If $v \in H^1_0(\Omega)$, then define $\chi(x,y) = \rho(x)$ where $\rho$ is smooth with compact support in $I_{\delta/2}$ and identically one on $I_\delta$.
Because $\supp(u) \subset \Omega_{\delta}$, we have 
$a_h(u, \chi\, v)=a_h(u,\chi\, v)$
and $q_h(u, \chi\,  v)=q_h(u,\chi\,v)$.
Since $\chi \cdot v \in H_1(\Omega_{\delta/2})$
the argument above gives
$$
\left | a_h(u, v)-q_h(u, v)\right|~
=~
\left | a_h(u, \chi\, v)-q_h(u, \chi\, v)\right|~
\leq~
C'\,
h\, 
a_h(u)^{\frac{1}{2}}\,
a_h(\chi\, v)^{\frac{1}{2}}.
$$
A direct computation shows that there exists 
a constant $C''$ so that $a_h(\chi\, w) \leq C''\,a(w)$ 
for each $w \in H^1_0(\Omega)$  when $h$ is small. 
The claim follows. 
\end{proof}  


\begin{proof}[{\bf The proof of Theorem \ref{thm:limits}}]
The proof is by contradiction. 
Assume that the conclusion is not true. Then  
there exists an eigenbranch $(E_h,u_h)$ 
such that $E_h$ tends to $E_0$ and $\sqrt{ E_0} \cdot L(0)/\pi$ is not an 
integer. 
Thus there exists $a<b$, $h_0>0$, and a  non-negative integer $k_0$, so that 
if $ h\leq h_0$, then $E_h \in [a,b]$ and 
\begin{equation}
\label{eqn:away-from-threshholds}
[a,b]~
\subset~
\biggr] 
\left(
\frac{k_0\,\pi}{L(0)}
\right)^2,\,
\left(
\frac{(k_0+1)\,\pi}{L(0)}
\right)^2
\biggr[\,.
\end{equation}

Because $E_h$ has a limit as $h$ tends to zero, 
its derivative $\partial_h E_h$ is integrable 
with respect to $dh$ near $h=0$.
Thus, it follows from 
(\ref{eqn-Feynman-Hellmann}) that there exists a 
subset $\Hbb \subset \op 0, \infty \cl$ with accumulation point $0$ and a sequence $(\eps_h)_{h\in \Hbb}$  such that $\eps_h \to 0$ and 
\begin{equation}
\label{eqn:integrability}
\|h \cdot \partial_x u_h\|^2\,\leq\,\eps_h\, \|u_h\|^2
\end{equation}
where $\|\cdot\|$ is the Euclidean $L^2$-norm on $\Omega$.
Our goal is to contradict (\ref{eqn:integrability}),
the {\it integrability condition}.

Suppose $h \leq h_0$. 
Then $E_h < b$ and so since $\| |\nabla u_h| \|^2= E_h\,\|u_h\|^2$
we find that $\| \partial_y u_h\|^2\,\leq\, b\, \|u_h\|^2$.
Therefore, it follows from Lemma \ref{lem:Poinc} 
that if $\delta>0$ is sufficiently small and $h \leq h_0$, then 
\[
\|u_h \cdot \un_{J_\delta}\|^2\,
\leq\,
\frac{1}{4}\,
\|u_h\|^2.
\]
We fix $\delta>0$ so that this estimate holds and so that the conclusion of 
Proposition \ref{prop:ncsum} holds.
Choose a cut-off function $\chi$ with support in $I_{\delta/2}$ and that is identically $1$ on 
$I_\delta$ and set 
\[
w_h\,
:=\,
\chi \cdot u_h
\mbox{\ \ \   and  \ \   }
r_h\,:=\,u_h-w_h.
\]
Since $|r_h|\,\leq \, |u_h| \cdot \un_{J_\delta}$, the preceding estimate gives 
$
\dis
\|r_h\|\,\leq\, \|u_h\|/2
$
and hence 
$
\dis 
\|w_h\|\geq  \|u_h\|/2
$.

Let $Q_h$ be the self-adjoint operator associated with $q_h$. We compute 
\[
Q_h (\chi \cdot u_h)\,
=\,
E_h\, \chi\, u_h\,+\,
[Q_h,\chi]\, u_h 
\mbox{\ \ \ and \ \ }
[Q_h,\chi]\,u_h
=\,
-h^2\,
\left(
2\chi'\, \partial_xu_h\,+\, \chi''\, u_h
\right).
\]
Hence we obtain 
\[
\| (Q_h-E)w_h\|\,\leq\, C\,  h\, \|u_h\|.
\]
Thus, if $\phi$ belongs to the form domain of $q_h$, then 
\[
\left| q_h(w_h,\phi) -E\, \langle w_h,\phi_h\rangle \right|\,
\leq\,
C\, h\, \|u_h\|\, \| \phi\|. 
\]
Using Lemma \ref{lem:asysep}, we then obtain  
\begin{equation*}%
\left| 
a_h(w_h,\phi)-E_h\, \langle w_h,\phi\rangle \right |\,\leq\,
C\, h \left( a_h(w_h)^\und\,+\,\|u_h\|\right)a_h(\phi)^\und.
\end{equation*}
The quantity $a_h(w_h)$ is majorized by a multiple of 
$q_h(u_h)= E_h\,\|u_h\|^2$, and  so  
\begin{equation}
\label{eqn:after-compare}
\left| a_h(w_h,\phi)-E_h\, \langle w_h,\phi\rangle \right |\,
\leq\,
C\, h\,  \|u_h\|\, a_h(\phi)^\und.
\end{equation}
We have $a_h(w_h, \phi) =\langle \Acal_h w_h, \phi\rangle$
where we regard $\Acal_h w_h$ as an element of $H^{-1}_{\Acal_h}$,
and moreover, 
${\|\phi\|^2_{H^{1}_{\Acal_h}} = a_h(\phi)+ \|\phi\|^2}$.
Therefore, from (\ref{eqn:after-compare}) we find that if $h \leq h_0$, then 
\begin{equation}
\label{eqn:H-minus-1-estimate}
\| (\Acal_h-E)w_h \|_{H^{-1}_{\Acal_h}}\,
\leq\, 
C\, h\, \|u_h\|. 
\end{equation}

Because of (\ref{eqn:away-from-threshholds})
and because each Fourier coefficient of $w_h$ lies in $\Cinf_0(I_{\delta})$,
we may apply Proposition \ref{prop:ncsum} to estimate $\|w_h\|$. 
By combining the resulting 
estimate with (\ref{eqn:H-minus-1-estimate}) we find that for 
$h \in \Hbb$ with $h \leq h_0$ 
\[
\| w_h\|\,
\leq\,
\eps\, \|u_h\| 
\,+\, 
C\,
\|h \cdot D'w_h\|
\]
where $D'$ is the operator defined in Proposition \ref{prop:ncsum}.\\
Using (\ref{eqn:D'}), we find that 
$$
D' w_h\,
=\,
\chi'\, u_h\, +\, \chi\, \px u_h\, -\, \chi\, \frac{L'}{L}\, \py u_h.
$$

Thus because $\chi$ has support in $I_\delta$, we get the estimate 
\[
\| h \cdot D' w_h\|\, 
\leq\,
C\left( \|h \cdot \partial_x u_h\|\,+\,h\, \|u_h\|\right).
\]
We finally obtain 
\[
\frac{1}{2}\,\| u_h\|\,\leq\,\|w_h\|\,
\leq\,
(C\, \eps\, \|u_h\|\,+\,C\, \|h \cdot \partial_x u\|\, \,+\, C\, h\, \|u_h\|). 
\]
We now can choose $\eps$ and $h_0$ small enough to absorb the first and last terms of the RHS into the LHS, 
yielding 
\[
\frac{1}{4}\,
\|u_h\|\,
\leq\, 
C\, 
\| h \cdot \partial_x u_h\|.
\]
This contradicts the integrability condition (\ref{eqn:integrability}). 
\end{proof}

\begin{thm}[Mixed conditions]
\label{thm:limits-mixed}
Let $E_h$ be an analytic eigenvalue branch of $q_h$ 
restricted $H^1_{0c}(\Omega)$ (or restricted to  $H^1_{0s}(\Omega)$)  relative to 
$\|\cdot\|^2$. Then there exists $k\in \N^*$ such that 
\[
\lim_{h \to 0}\,  E_h~ =~ \left( \frac{(k-\und)\pi}{L(0)}\right)^2.
\]
\end{thm}

\begin{proof}
The proofs are nearly identical to the proof of Theorem \ref{thm:limits}.
In the case where $q_h$ is restricted to $H^1_{0c}(\Omega)$,
we use the Fourier decomposition 
$$
u(x,y)\,
=\,
\sum_{k\geq 1} u_k(x) \cdot \sqrt{2}\, \cos \left(\frac{(k-\und)\pi y}{L(x)} \right)
$$
In the case where $q_h$ is restricted to $H^1_{0s}$, 
the identification of $L^2(\Omega)$ with $\ell^2(\N^*, \Hcal)$ is 
given by 
$$
u(x,y)\,
=\,
\sum_{k\geq 1} u_k(x) \cdot \sqrt{2}\, \sin \left(\frac{(k-\und)\pi y}{L(x)} \right).
$$
In each case, an analogue of the Poincaré estimate---Lemma \ref{lem:Poinc}---still holds as does an analogue of Lemma \ref{lem:asysep}. 
Using the same method used to prove Theorem \ref{thm:limits}, 
we obtain the theorem.
\end{proof}


\newpage 
\section{Eigenvalue limits for symmetric domains with Dirichlet conditions}
\label{sec:symmetric}

In this section, we consider domains $\Omega  \subset \R^2$ of the form 
$$
\Omega \,=\,\left \{ (x,y)\in \R^2,~~ |x|\,<\,1,~~ -L(x)\, <\,y\,<\,L(x) \right\}
$$
where $L$ satisfies conditions L1 and L2 of section \ref{sec:interval}.
We consider the family of quadratic forms $q_h$ defined as in (\ref{eqn:quad})
and restricted to $H^1_0(\Omega)$. 

Note that $\Omega$ is invariant under the reflection symmetry $(x,y) \to (x,-y)$.
This symmetry defines an orthogonal decomposition,
$L^2(\Omega, dx\, dy)=L^2_{\rm odd}(\Omega) \oplus L^2_{\rm even}(\Omega)$,
into the sum of the space of `odd' functions and the space of `even' functions. 
If $u, v \in H^1_0(\Omega)$ and $u$ and $v$ have opposite parity, then $q_h(u,v)=0$.
It follows that each eigenspace $V$ of $q_h$ on $H^1_0(\Omega)$ with respect 
to $\|\cdot\|_{L^2}^2$  equals 
$(V \cap L^2_{\rm odd}(\Omega)) \oplus (V \cap L^2_{\rm even}(\Omega))$.
In particular, the spectral problem for $q_h$ on  $H^1_0(\Omega)$ reduces 
to the study of $q_h$ on $H^1_{0}(\Omega) \cap L^2_{{\rm odd}}(\Omega)$
and the study of $q_h$ on $H^1_{0}(\Omega) \cap L^2_{{\rm even}}(\Omega)$.

The Kato-Rellich theory \cite{Kato} applies separately to the family $q_h$ restricted 
to `odd' functions and the family $q_h$ restricted to `even' functions. 
The theory provides analytic paths $h \mapsto u_{h,\ell} \in L^2_{{\rm odd}}(\Omega)$
(resp.  $L^2_{{\rm even}}(\Omega)$)  and $h \mapsto E_{h,\ell}$ so that for each 
the set $\{u_{h,\ell}:\, \ell \in \N^*\}$ is a Hilbertian basis 
for $L^2_{{\rm odd}}(\Omega)$ (resp. $L^2_{{\rm even}}(\Omega)$)
and 
$q_h(u_{h,\ell},v) = E_{h,\ell}\, \langle u_{h,\ell}, v \rangle$ for $v \in H^1_0(\Omega)$.

\begin{thm}
\label{thm:dichotomy}
If $E_h$ is an `odd' analytic eigenvalue branch of $q_h$ restricted 
to $H^1_{0}(\Omega) \cap L^2_{{\rm odd}}(\Omega)$, then 
\[
\lim_{h \to 0}\,  E_h~ =~ \left( \frac{k\, \pi}{L(0)} \right)^2.
\]

If $E_h$ is an `even' analytic eigenvalue branch of $q_h$ restricted 
to $H^1_{0}(\Omega) \cap L^2_{{\rm even}}(\Omega)$, then 
\[
\lim_{h \to 0}\,  E_h~ =~ \left( \frac{ \left(k-\und \right)\pi}{L(0)}\right)^2.
\]
\end{thm}

\begin{proof}
This follows from the results of \S \ref{sec:shrink}. Indeed, let 
$\Omega'$ denote the intersection of $\Omega$ with the upper half plane $\{y>0\}$.
Restriction  $u \mapsto u|_{\Omega'}$ defines an isomorphism 
from the space of `odd' $H^1_0(\Omega)$ functions  
onto the space $H^1_0(\Omega')$ 
This restriction also  defines an isomorphism 
from the space of `even' $H^1_0(\Omega)$ functions  
onto the space  $H^1_{0c}(\Omega')$ of functions whose trace along the graph of $L$ vanishes identically.
Moreover, the ratio $q_h(u)/\|u\|^2$ is unchanged by restriction. 
The claim then follows from Theorems \ref{thm:limits} and \ref{thm:limits-mixed}.
\end{proof}

We next combine \ref{thm:dichotomy} and Bourget's hypothesis (Appendix \ref{sec:disk}) to prove the generic simplicity of ellipses.
First note that, up to isometry, each ellipse has the form 
$$
\Ecal_{a,b}\,
:=\,
\{ (x,y)\, :\, (x/a)^2 + (y/b)^2 \leq 1 \}.
$$
Thus, we may naturally identify the set of isometry classes of ellipses 
with the set of $\{(a,b)\, :~ 0 < b \leq a < \infty\}$. 

\begin{thm}[Theorem \ref{thm:main}]
\label{thm:ellipse-simple}
There exists a countable subset $\Ccal \subset \op 0,1]$ 
so that if $b/a \notin \Ccal$, then each eigenvalue of the 
Dirichlet Laplace operator of the ellipse $\Ecal_{a,b}$ is simple. 
\end{thm}

\begin{proof}
The multiplicities of the Laplace spectrum of a domain are unchanged by 
a homothety of the domain.
In particular, the spectrum of $\Ecal_{a,b}$ is simple if and only if the 
spectrum of $\Ecal_{a/b,1}$.   Let $\Omega$ be the unit disk. 
If we let $\Omega_h$ denote the image of $\Omega$ 
image under $(x,y) \mapsto (x/h,y)$, then ${\Ecal_{1/h,1}= \Omega_h}$. 
As discussed in Remark \ref{remk:stretch}, the study of the 
Dirichlet Laplace spectrum of $\Omega_h$ is equivalent to the study of 
the quadratic form $q_h$ of (\ref{eqn:quad}) restricted to $H^1_0(\Omega)$
relative to $\|\cdot\|^2$. Thus, to prove the claim,
it suffices to show that there exists a countable set $\Ccal \subset \op 0, 1]$
such that $q_h$ has simple spectrum.

To `construct' the set $\Ccal$, it suffices to show that each of the 
real-analytic eigenvalue branches of $q_h$ is distinct. 
Indeed, let $\{ h \mapsto E_{h, \ell}\, : \, \ell \in \N^* \}$ denote the 
collection of analytic eigenvalue branches associated to the family $q_h$.
If for some $h$ we have $E_{h, \ell} \neq E_{h, \ell'}$ then 
analyticity implies that the set 
$\Ccal_{\ell, \ell'}:=\{h\, :\, E_{h, \ell} =  E_{h, \ell'}\}$
is countable. Thus, $\dis \Ccal = \bigcap_{\ell \neq \ell'} \Ccal_{\ell, \ell'}$
is the desired set.

To prove that the various analytic eigenvalue branches are distinct, we argue by contradiction. Suppose to the contrary, that $\ell \neq \ell'$ but 
$E_{h,\ell}= E_{h, \ell'}$ for each $h \in \op 0,1]$. 
Let $u_{h,\ell}$ and $u_{h,\ell'}$ denote the corresponding 
analytic eigenfunction branches. For each $h$ we have $u_{h,\ell} \perp u_{h,\ell'}$
and in particular $u_{1,\ell} \perp u_{1,\ell'}$.
It follows from the discussion before Theorem \ref{thm:dichotomy}
that either $u_{h, \ell}$ (resp. $u_{h, \ell'}$) is `odd' for each $h$
or  $u_{h, \ell}$ (resp. $u_{h, \ell'}$) is `even' for each $h$.

It is well-known that each eigenspace $V$ of $q_1$ is at most two dimensional,
and moreover, if $\dim(V)=2$, then 
$\dim(V \cap L^2_{{\rm odd}}(\Omega))=1= \dim(V \cap L^2_{{\rm even}}(\Omega))$
(see Corollary \ref{thm:struct} in Appendix \ref{sec:disk}).  Therefore, $u_{1,\ell}$ and  $u_{1,\ell'}$
span an eigenspace of $q_1$ and they have opposite parity. 
Hence they have opposite parity for each $h$. 
Theorem  \ref{thm:dichotomy} then implies 
$\dis \lim_{h \to 0} E_{h, \ell} \neq  \lim_{h \to 0} E_{h, \ell'}$,
a contradiction.
\end{proof}

We now remark that this proof also yields a way to construct ellipses that have some multiplicity in the spectrum. We obtain the following proposition.

\begin{prop}
\label{prop:mult}
    There exists a sequence $(h_n)_{n\geq 0}$ going to zero such that, for any $n$, if $b/a=h_n$ then the ellipse $\Ecal_{a,b}$ has at least one multiple eigenvalue.
\end{prop}

\begin{proof}
Fix some $h_0$, it suffices to show that there exists $h<h_0$ so that the spectrum of $\Ecal_{1/h,1}$ has (at least) one eigenvalue of multiplicity at least $2$.
Let $h\mapsto E_{h,\mathrm{odd}}$ be the smallest odd eigenvalue of $\Ecal_{h,1}$. This map is continuous, piecewise analytic, non-decreasing and it tends to $\pi^2$. 
Using, min-max theory, it is easy to show that, for any $N$ the 
$N^{{\rm th}}$ eigenvalue of $\Ecal_h$ converges to $\frac{\pi^2}{4}$. 
Let now $N$ be such that the $N^{{\rm th}}$ eigenvalue of $\Ecal_{h_0,1}$ is greater than $E_{h_0,\mathrm{odd}}$. By continuity, there must be some $h<h_0$ such that the 
$N^{{\rm th}}$ eigenvalue of $\Ecal_{1/h,1}$ is equal to $E_{h,\mathrm{odd}}$. 
\end{proof}


\newpage

\section{Ellipsoids}
\label{sec:ellipsoids} 

In this section, we prove that the generic ellipsoid has simple spectrum. As for ellipses, this result relies 
on knowing precisely how the multiplicities in the spectrum of the unit ball are distributed and on the limiting 
behaviour of analytic eigenbranches upon the stretching/compressing degeneration. 

We denote by $\Ebb_{a,b,c}$ the ellipsoid:
$$
\Ebb_{a,b,c}~ 
=~
\{(x,y,z) \in \R^3\, :\, (x/a)^2 + (y/b)^2 +(z/c)^2 < 1\}.
$$

We first consider the ellipsoids with `circular cross-sections': 
$$
\Ebb_{1/h,1,1}~ 
=~
\{(x,y,z) \in \R^3\, :\, (h\,x)^2 + y^2 +z^2 < 1\}.
$$

\begin{prop}
\label{prop:circular}
There exists a countable subset $\Ccal \in \op 0,1 ]$ such that if $ h \notin \Ccal$,
then each eigenspace of the Dirichlet Laplacian on $\Ebb_{1/h,1,1}$ is at most two dimensional. 
Moreover, if $V$ is a 2-dimensional eigenspace of $\Ebb_{1/h,1,1}$, then there exist
$u_{\pm} \in V$ so that $V = {\rm span} \{u_+,u_-\}$ and $u_{\pm}(x,y,-z)= \pm u(x,y,z)$.
\end{prop}

The proof of Proposition \ref{prop:circular} will use the same
ideas as the proof of Theorem \ref{thm:ellipse-simple} in a slightly different setting.
First, by a  change of variables, we are led to consider
the following quadratic form
\begin{equation}
\label{eqn:D-form}
\dis 
q_h(u)
:=
\int_B 
\left(h^2  \cdot |\partial_x u|^2\, +\, |\partial_y u|^2\, +\, |\partial_z u|^2 \right)\, 
dx\, dy.
\end{equation}
defined on $H^1(B)$ where $B$ is the unit ball in $\R^3$.

Instead of using a reflection symmetry to decompose the spectrum, we will 
use the natural rotational symmetry. To be concrete, it will prove convenient 
to use cylindrical coordinates $(x,r,\theta)$ where $(y,z)= (x, r \cos(\theta), r \sin(\theta))$. Then the rotation about the $x$-axis 
through angle $\alpha$ is given by 
$$  
R_{\alpha}(x,r,\theta)~
=~
(x, r, \theta + \alpha).
$$
The map $u \mapsto u \circ R_{\alpha}$ is an isometry of both $L^2(B)$ and $H^1_0(B)$.
Note that $q_h(u \circ R_{\alpha})= q_h(u)$ for each $u \in H^1(B)$.
For $m \in \Z$ define 
$$
L^{2}_{m}(B)~
=~
\{ 
u \in L^2(B)\,
:\,
u \circ R_{\alpha}\,
=\,
e^{i m \alpha} \cdot u 
\}
$$
\begin{equation}
\label{eqn:H_m}
H^{1}_{0,m}(B)~
=~
\{ 
u \in H^1_0(B)\,
:\,
u \circ R_{\alpha}\,
=\,
e^{i m \alpha} \cdot u 
\}
\end{equation}
Then 
$\dis L^2_{0}(B)= \bigoplus_m L^2_{m}(B)$ and 
$\dis H^1_{0}(B)= \bigoplus_m H^1_{0,m}(B)$.

We fix $m \in \Z$, and let $q_{h,m}$ denote the restriction $q_h$ to $H^1_{0,m}(B)$.

Let $L(x)= \sqrt{1-x^2}$, and let $D_x$ denote the disk $\{(r, \theta)\, :\, r < L(x)\}$. 
The rotation $R_{\alpha}$ preserves $D_x$, and we define the spaces
$L^2_{m}(D_x)$ and  $H^1_{0,m}(D_x)$ as above.
Let $\{\phi^m_{\lambda}\}$ be a Hilbertian basis of $H^1_{0,m}(D_0)$ 
consisting of eigenfunctions of the Dirichlet Laplacian on the unit disk 
$D_0$.\footnote{Each $\phi_{\lambda}^m$ is a multiple of $J_m(\sqrt{\lambda}\, r)\,e^{im\theta}$ where $J_m$ is the Bessel function of order $m$.}
Let $\spec_{m}(D_0)$ denote the set of eigenvalues associated with $\{\phi^m_{\lambda}\}$.

As with the case of ellipses, the proof of
Proposition \ref{prop:circular} is based on a
determination of the limits of analytic eigenbranches:

\begin{prop}
\label{prop:limits-ball}
For any analytic eigenvalue branch $E_h$ of $q_{h,m}$ on $H^1_{0,m}(\Omega)$ relative to 
$\|\cdot\|^2$, there exists $\lambda \in \spec_m(D_0)$ such that 
\[
\lim_{h \to 0}\,  E_h~ =~  \frac{\lambda}{L(0)^2}
\]
\end{prop}

\begin{proof}
The proof follows the same outline as the proof of Theorem \ref{thm:limits}.
We provide a sketch and leave the details to the reader. 
For $u \in H^1_0(B)$ with compact support we have 
\begin{equation}
\label{eqn:eigen-expansion}
u(x,r,\theta)~
=~
\sum_{\lambda \in \spec_m(D_0)}\,
u_{\lambda}(x)\
\cdot\,
\phi^m_{\lambda}\left( \frac{r}{L(x)},\, \theta \right).
\end{equation}
Given $\delta \in \op 0, 1 \cl$, let 
$B_{\delta}= \{(x,y,z) \in B\, :\, -1+\delta < x < 1-\delta \}$.  
By using (\ref{eqn:eigen-expansion}) in place of the Fourier sine decomposition,
one proves a Poincar\'e estimate on $B \setminus B_{\delta}$ 
analogous to Lemma \ref{lem:Poinc} 
with $\py$ replaced by $\nabla = \pr u\, \pr + r^{-2} \ptheta u\, \ptheta$.

For $u \in H^1_{0,m}(B)$ define 
$$ 
(D' u)(x,r, \theta)~
=~
\sum_{\lambda \in \spec_m(D_0)}\,
u_{\lambda}'(x)\,
\cdot\,
\phi^m_{\lambda}\left( \frac{r}{L(x)},\, \theta \right).
$$
and
$$
a_{h,m}(u)~
=~
\|h \cdot D'u\|^2\,
+\, 
\| \nabla  u\|^2
$$
where $\nabla u= \pr u\, \pr + r^{-2} \ptheta u\, \ptheta$. 
By using the method of proof of Lemma \ref{lem:asysep}, 
one finds that $q_{h,m}$ and $a_{h,m}$ are asymptotic at first order
in the sense of (\ref{eqn:asysep}) where $\Omega_{\delta}$ is 
replaced by $B_{\delta}$.

Following section \ref{sec:sum}, one  now 
defines an operator $\Acal_{h,m}$ associated with the quadratic form $a_{h,m}$ 
by identifying $e_k$ with $\phi^m_{\lambda_k}$
where $\lambda_k$ is the $k^{{\rm th}}$ element of $\spec_{m}(D_0)$. 
The Hilbert space $\Hcal$ is now $\Hcal=L^2( ]-1,1[, L^2(x)dx)$ and the operators $A_{h,m,k}$ 
are defined as 
\[
A_{h,m,k}v\,=\,-\frac{1}{L^2}(L^2 v')'\,+\,\frac{\lambda_k}{L^2}v.
\]

One proves the analogue of Proposition \ref{prop:ncsum} with 
$(k_0 \pi)^2$ and $((k_0+1) \pi)^2$ replaced by successive elements
$\lambda_{k_0} < \lambda_{k_0+1}$ in $\spec_m(D_0)$.

Given the analogues of Proposition \ref{prop:ncsum}, 
Lemma \ref{lem:Poinc}, and Lemma \ref{lem:asysep}, the proof of \ref{prop:limits-ball}
follows the same lines as does the proof of Theorem \ref{thm:limits}.
\end{proof}

\begin{proof}[\bf Proof of Proposition \ref{prop:circular}]
Given Proposition \ref{prop:limits-ball} and well-known facts 
about the Dirichlet spectrum of the unit ball, much of the proof 
follows the same outline as the proof of Theorem \ref{thm:ellipse-simple}.
As described there, the `construction' of the set $\Ccal$ 
will follow from proving that the various
analytic eigenvalue branches of $q_h$ are all distinct. 

In the present case, Kato-Rellich theory implies that 
there exist analytic branches $h \mapsto E_{h,\ell}^m \in \R$ and 
$h \mapsto u_{h,\ell}^m \in H^1_{0,m}$ so that, for each $h \in \op 0,1]$, the set 
$\{u_{h, \ell}^m\, :\,   \ell \in \N^* \}$ is a Hilbertian basis 
for $L^2_m(B)$ and so that for each $\ell \in \N^*$ the 
function $u_{h,\ell}^m$ is an eigenfunction 
of $q_{h,m}$ relative to $\|\cdot \|^2$ with eigenvalue $E_{h,\ell}^m$.
Bourget's hypothesis implies 
that if $|m| \neq |m'|$ then  $\spec_m(D_0) \cap \spec_{m'}(D_0) =\emptyset$ 
(see Appendix \ref{sec:disk}). Thus Lemma \ref{prop:limits-ball}
implies that if $|m| \neq |m'|$, then for every $\ell, \ell' \in \N^*$,
the branch $E_{h,\ell}^m$ is distinct from the branch $E_{h,\ell'}^{m'}$. 

Suppose for contradiction
that $q_h$ has two distinct analytic eigenfunction branches $u_h$ and $u_h'$
whose corresponding eigenvalue branches coincide for all $h$. 
Then the preceding paragraph implies that there exist $m$, $m'$, $\ell$, and $\ell'$
so that $|m|=|m'|$, $u_h = u_{h,\ell}^m$ and $u_h' = u_{h,\ell'}^{m'}$. 

The functions $u_1$ and $u_1'$ lie in the same eigenspace $V_1$ of 
the Dirichlet Laplace operator on the unit ball.
The intersection of $H^1_{0,m}(B)$ with any such eigenspace 
is one-dimensional (see Appendix \ref{sec:disk}). 
Hence we cannot have $m' = m$. Therefore,  $m' = -m$,
the function $u_1$ spans $V_1 \cap H^1_{0,m}(B)$, and $u_1'$ 
spans $V_1 \cap H^1_{0,-m}(B)$.

From (\ref{eqn:H_m}), one sees that complex conjugation defines 
a unitary isomorphism from $H^1_{0,m}(B)$ onto $H^1_{0,-m}(B)$.
Because $q_h(\overline{v})= q_h(v)$, if $(E_h,u_h)$ is an analytic eigenbranch 
for $q_{h,m}$, then $(E_h,\overline{u}_h)$ is an analytic 
eigenbranch for $q_{h,-m}$. In particular, we may assume that 
$(E_{h,\ell}^{-m}, u_{h,\ell}^{-m})= (E_{h,\ell}^{m}, \overline{u}_{h,\ell}^{m})$
for each $\ell$, $h$, and $m$.

Since $u_1'$ spans $V_1 \cap H^1_{0,-m}(B)$,
it follows that $u_1' = c \overline{u}_1$ for some constant $c$.  
Hence, without loss of generality, $u_h'= \overline{u}_h$ for each $h$.
It follows that $\ell=\ell'$ and 
one can find real eigenfunctions  $u_+$ and 
$u_- \in V_1 \cap (H^1_{0,m}(B) \oplus  H^1_{0,-m}(B))$
that satisfy $u_{\pm}(x,y,-z)= \pm u_{\pm}(x,y,z)$.
\end{proof}

\begin{thm} 
There exists a set of $\Ccal \subset \op 0, 1] \times \op 0, 1]$ of
Lebesgue measure zero so that if $(h, h') \notin \Ccal$, then 
each eigenspace of the Dirichlet  Laplacian on $\Ebb_{1/h,1/h',1}$ is 
one-dimensional.

\end{thm}

\begin{proof}
Let $\Ecal_h$ denote the ellipse $\{ (y, z)\, :\,  (h\, y)^2 + z^2 < 1\}$.
By Theorem \ref{thm:ellipse-simple} and 
Proposition \ref{prop:circular}, there exists $h_0 \in \op 0, 1]$  so that 
the Dirichlet spectrum of $\Ecal_{h_0}$ is simple and so that
each Dirichlet eigenspace $V$ of the ellipsoid 
$\Ebb_{1,1/h_0,1}$ is at most 2-dimensional (exchanging the $x$ and $y$ coordinates is harmless), and moreover,
if $\dim(V)=2$, then $V={\rm span} \{u_+,u_-\}$ where $u_{\pm}(x,y,-z)= \pm u(x,y,z)$.
We fix $h_0$ and consider the family of domains $\Omega_{h}:= \Ebb_{1/h,1/h_0,1}$. 
It suffices to show that there exists some $h \in \op 0, 1]$ so that $\Omega_h$ has 
simple Dirichlet Laplace spectrum. 
Indeed, analyticity of the spectrum would then imply that for each line $\ell$
in $\op 0, 1] \times \op 0, 1]$ that passes through $(h,h_0)$, there exists
a countable subset $\Ccal_{\ell} \subset \ell$ such that if 
$(g, g') \in \ell \setminus \Ccal_{\ell}$, then $\Ebb_{1/g,1/g',1}$ has simple spectrum. 
Then $\Ccal = \cup_{\ell}\, \Ccal_{\ell}$ would be the desired set of measure zero.  

Let $\Omega:= \Omega_{1}=\Ebb_{1,1/h_0,1}$ 
and consider the quadratic form $q_h$ of (\ref{eqn:D-form})
on $H^1_0(\Omega)$. Let $q_{h,\pm}$ denote the restriction of $q_h$ to 
the space $H^1_{0,\pm}(\Omega)= \{ u \in H^1_0(\Omega)\, :\, u(x,y,-z)= \pm u(x,y,z)\}$.
Let  $\{\phi_{\lambda}^{\pm}\}$ be a Hilbertian basis of 
Dirichlet eigenfunctions 
of the ellipse $\Ecal_{h_0}$
such that $\phi_{\lambda}^{\pm}(y,-z)= \pm\, \phi_{\lambda}^{\pm}(y,z)$. 
Let $\spec_{\pm}(\Ecal_{h_0})$ denote the set of (simple) eigenvalues associated to the 
eigenfunctions $\phi^{\pm}_{\lambda}$.
The now familiar method that was used to prove Theorem \ref{thm:limits} 
and Proposition \ref{prop:limits-ball} may be used to prove 
that each real-analytic eigenvalue branch $E_h$ of $q_{h, \pm}$ 
limits to an element of $\spec_{\pm}(\Ecal_{h_0})$.

The idea used to prove Theorem \ref{thm:ellipse-simple}
can now be used to prove that there exists a countable set $\Ccal' \in \op 0,1]$ so that 
if $h \notin \Ccal'$, then the spectrum of $q_h$ (and hence $\Omega_h$) is simple. 
Indeed, each analytic eigenvalue branch of $q_h$ on $H^1_0(\Omega)$ is either 
an eigenvalue branch of $q_{h,-}$ or an 
eigenvalue branch of $q_{h,+}$.  Because $\spec(\Ecal_{h_0})$ is simple, we have 
$\spec_{-}(\Ecal_{h_0}) \cap \spec_{+}(\Ecal_{h_0})= \emptyset$.
Therefore, since the $q_{h,\pm}$
branches limit to an element of $\spec_{\pm}(\Ecal_{h_0})$, 
the $q_{h,-}$ branches are distinct from the 
$q_{h, +}$ branches.  In particular, if two distinct 
eigenfunction branches $u_h$ and $u_{h'}$ of $q_h$ were to 
have coincident eigenvalue branches, then either each branch is a branch of $q_{h,-}$
or each is a branch of $q_{h,+}$. Thus, $u_1$ and $u_1'$ would either  be
two independent `odd' Dirichlet eigenfunctions of 
$\Ebb_{1,1/h_0,1}$ or would
be two independent `even' Dirichlet eigenfunctions of $\Ebb_{1,1/h_0,1}$.
This would contradict the choice of $h_0$ made above. 

Therefore, no two analytic eigenvalue branches of $q_h$ coincide, 
and so analyticity provides the existence of the desired set $\Ccal'$.
The spectrum of $q_h$ on $\Omega$ is the same as the spectrum of
the Dirichlet Laplacian on $\Ebb_{1/h,1/h_0,1}$, and thus the latter is simple
for some $h \in \op 0, 1]$.
\end{proof}

\begin{remk}
We observe that the strategy of the present paper could apply to study the quadratic form $q_h$ on a domain $\Omega \subset \R^N$ 
that satisfy the following property for a profile function $L$:
\[
(x_1,\cdots,x_N)\in \Omega 
\iff 
\left(
0, \frac{x_2}{L(x)}, \cdots,\frac{x_N}{L(x)}
\right)
\in 
\Omega.
\]
\end{remk}

\newpage 


\appendix
\section{The spectrum of the unit disk and the unit ball}
\label{sec:disk}
The Dirichlet and Neumann spectra of the unit disk $D$ may 
be computed using separation of variables and Bessel functions. 
A Hilbertian basis of eigenfunctions of $q_1$ restricted to 
$H^1_0(D) \cap L^2_{{\rm odd}}(D)$ is given by $J_k(\sqrt{E}\, r)\sin(k\theta)$ where $J_k$ is the standard Bessel and $\sqrt{E}$ is a zero of $J_k.$
A Hilbertian basis of eigenfunctions of $q_1$ on $H^1(\Omega)(D) \cap 
L^2_{{\rm even}}(D)$
is given by $J_k(\sqrt{E}\, r)\cos(k\theta).$ 
It follows that the part of the spectrum of $q_1^N$ and $q_1^D$ that correspond to positive $k$ coincide, so that 
the multiplicity of an eigenfunction that has a non-radial corresponding eigenfunction is at least $2$.

The following theorem is a well-known corollary of Siegel's work \cite{Siegel}.

\begin{thm}[Bourget's hypothesis]
\label{thm:Siegel}
If $k$ and $k'$ are distinct integers, then the Bessel functions $J_k$ and $J_{k'}$ 
have no common zeros other than $0$. 
\end{thm}  
 
\begin{proof}
See \S 15.28 of \cite{Watson}.\footnote{Bourget's hypothesis was presented as a
conjecture in the first edition of \cite{Watson}, and later was presented as a theorem in the second edition.}
If $k'>k$, then a recurrence relation implies 
that $J_{k'}(z)= p(1/z)\, J_k(z) + q(1/z)\, J_{k-1}(z)$
where $p,q$ are polynomials\footnote{The polynomials $p$ and $q$ are called {\it Lommel polynomials}.}
with rational coefficients.
It is known that $J_k$ and $J_{k-1}$ have no common zeros other than zero \cite{Watson}.
Thus, if $J_{k}$ and $J_{k'}$ were to have a common zero $z_0 \neq 0$, then 
$q(1/z_0)=0$.
Hence $z_0$ is an algebraic number and Siegel's theorem \cite{Siegel}
would imply that $J_k(z_0)$ is not algebraic.
But $J_k(z_0)=0$, a contradiction.
\end{proof}

Bourget's hypothesis has the following corollary.
Recall from \S \ref{sec:symmetric}, the notion of `even' and `odd' function
with respect to the symmetry $(x,y) \mapsto (x, -y)$.

\begin{coro}
\label{thm:struct}
The dimension of each eigenspace $V$ of $q_1$ on $H^1(D)$ is at most two.
If $dim(V)=1$, then each $u \in V$ is a radial function (hence 'even'). 
If $dim(V)=2$, then $V$ is spanned by an `even' function and an `odd' function.
\end{coro} 

\begin{proof}
For each $k$, let $Z_k$ be the set of $E>0$ so that $J_{k}\left(\sqrt{E}\right)=0$.
Theorem \ref{thm:Siegel} implies that the sets $Z_k$ are disjoint.
By the discussion before Theorem \ref{thm:Siegel}, the spectrum of $q_1$ 
equals the union $\bigcup_{k\geq 1} Z_k$. Moreover, if $k=0$, then for 
each $E \in Z_k$, the space is spanned by $J_0(\sqrt{E} \cdot r)$
which is radial. If $k>0$ and $E \in Z_k$, the space is spanned 
by $J_k(\sqrt{E} \cdot r) \sin(\theta)$
which `odd' and $J_k(\sqrt{E} \cdot r) \cos(\theta)$ which is `even'.
\end{proof}

To describe the eigenspaces of the Dirichlet Laplace operator on the unit 
ball $B \subset \Rbb^3$, we use the harmonic polynomials $P_{\ell}$
of degree $\ell$ on $\Rbb^3$. To describe a useful basis, for $P_{\ell}$
we use the differential operators associated to rotating about 
the $x$, $y$, and $z$-axes:
\begin{eqnarray*}
R_x &=& -z \partial_y + y\partial_z, \\
R_y &=& -x \partial_z + z\partial_x,  \\
R_z &=& -y \partial_x + x\partial_y.
\end{eqnarray*}
The `ladder operators'  $L:= R_y + i R_z$ and $\oL := R_y - i R_z$
preserve $P_{\ell}$. 
A computation shows that the polynomial $(y-iz)^{\ell}$ is harmonic. 
If we define $Y_{\ell,k}:= L^k (y-iz)^{\ell}$, then for 
each $k$, we have 
\begin{equation}
\label{eqn:Y-eigen}
R_x Y_{\ell,k}= -i(\ell-k) \cdot Y_{\ell,k}.
\end{equation}
Moreover, the set 
$$
\dis
\left\{
Y_{\ell,k}\,
:\,
k= 0, 1,\ldots, 2 \ell
\right\}
$$
is a basis for $P_\ell$.

In terms of  spherical coordinates $(r, \omega)$ for $B$, 
a  Hilbertian basis of Dirichlet eigenfunctions 
is given by 
$$
\phi_{\lambda,\ell,k}(r,\omega)~
:=~
r^{-\und}\, J_{\ell+ \und}(\sqrt{\lambda} r) \cdot Y_{\ell,k}(\omega)
$$
where $\ell \in \N$, $k= 0, \cdots, 2\ell$, and 
$J_{\ell+ \und}(\sqrt{\lambda})=0$. Bourget's hypothesis also 
holds for Bessel functions of half-integer order, and so
the nonzero zeros of $J_{\ell+ \und}$ 
are distinct from the nonzero zeros of $J_{\ell'+ \und}$ 
if $\ell \neq \ell'$.  The eigenspace $V$ associated 
to the eigenvalue $\lambda$ where $J_{\ell+ \und}(\sqrt{\lambda})=0$ has basis 
$\{\phi_{\lambda,\ell,k}\, :\, k=0, \ldots, 2 \ell \}$.

Recall the subspace $H^1_{0,m}(B)$ defined in $(\ref{eqn:H_m})$. 

\begin{lem}
Let $V$ be an eigenspace of the Dirichlet Laplace operator on the unit ball
associated to a zero of $J_{\ell + \und}$. For $|m|\leq \ell$, the 
dimension of $V \cap H^1_{0,m}(B)$ equals one. 
\end{lem}

\begin{proof}
It follows from (\ref{eqn:Y-eigen}) that the function $Y_{\ell,\ell+m}$ 
belongs to $H^1_{0,m}(B)$. 
\end{proof}

\newpage

\section{The Neumann case}
\label{sec:neu}

The purpose of this appendix is to prove the analogues of Theorems \ref{thm:limits} 
and \ref{thm:limits-mixed} in the full Neumann
case. Namely, we prove the following theorem, 
in which we use the notation of section \ref{sec:shrink}. 

\begin{thm}[Full Neumann]
\label{thm:limits-Neumann}
For any analytic eigenvalue branch $E_h$ of $q_h$ on $H^1(\Omega)$ relative to 
$\|\cdot\|^2$, there exists $k\in \N\,=\,\N^*\cup \{0\}$ such that 
\[
\lim_{h \to 0}\,  E_h~ =~ \left( \frac{k\, \pi}{L(0)} \right)^2.
\]
\end{thm}

As is obvious from the statement, the main difference between the full Neumann case and the boundary conditions considered in \S \ref{sec:shrink} is the
presence of a zeroth mode corresponding to $k=0$. Of course, the presence of this
mode follows from the fact that constant functions are in
the kernel of the Neumann Laplace operator in $\Omega$. Given  $v\in L^2(\Omega)$, we perform the transversal Fourier decomposition:

\[
  v\,=\,v_0\otimes \un \,+\,\sum_{k\geq 1} v_k\otimes e_k, 
\]
where 
\begin{eqnarray*}
  e_k(x,y)&=& \sqrt{2}\cos \frac{k\pi y}{L(x)}\\
 v_k(x)&=& \frac{1}{L(x)}\int_0^{L(x)} v(x,y)e_k(x,y)\,dy,\\
    v_0(x) &=& \frac{1}{L(x)}\int_0^{L(x)}v(x,y)\, dy.
\end{eqnarray*}

To prove Theorem \ref{thm:limits-Neumann}, we apply the same overall 
strategy as in the proof of Theorem \ref{thm:limits}. 
That is, we suppose for contradiction that a (normalized) analytic 
eigenbranch $(E_h,u_h)$ does not limit to one of the thresholds $(\pi k)^2/L(0)^2$, 
and we seek to contradict the resulting integrability condition (\ref{eqn:integrability}).
As we will see, the arguments that we used in the central region
$\Omega_\delta$ still work so that the integrability condition 
gives that, 
for any $\delta$ and any `horizontal' cutoff function
$\chi_{\delta}$, the quantity 
$\| \chi_\delta u_h\|$ is eventually, for $h$ small,  arbitrarily small.
This will imply that $u_h$ concentrates away of $\Omega_\delta$.
The Poincar\'e inequality is also still valid for $k\geq 1$ so that, 
for any $\delta$, the quantity  $\|(1-\chi_\delta)(u_h-u_{h,0})\|$ also tends to $0$.
As a consequence, it will follow that $u_h$ eventually has all of its mass supported on its zeroth mode near the points $x={\pm 1}$.
We will obtain a contradiction by making a careful study of the second order inhomogenous ODE satisfied by $u_{h,0}$.

\subsection{In the central region $\Omega_\delta$}\hfill \\
Recall the region $\Omega_{\delta}$ defined in \S \ref{sec:shrink}.
In this section, we fix $\delta \in (0,\frac{1}{2})$ and define $H^1_v(\Omega_\delta)$ to be the subspace of functions
in $H^1$ that vanish on the vertical sides $x=1-\delta$ and $x=-1+\delta$.

For each  $v\in H^1_v(\Omega_\delta)$,
we define
\[
  D'(v)\,=\,v'_0\otimes \un \,+\,\sum_{k\geq 1} v_k'\otimes e_k,
\]
and the quadratic form
\[
  a_h(v)\,=\,\sum_{k\geq 0} \int_{-1+\delta}^{1-\delta} \left[ h^2|v_k'(x)|^2\,+\,\frac{k^2\pi^2}{L^2(x)}|u(x)|^2 \right]\, L(x)dx.
\]
For $h$ small enough, using that $L'$ and $\frac{1}{L}$ are uniformly bounded on $\Omega_{\delta}$, we obtain the following lemma that implies
that $q_h$ and $a_h$ are asymptotic at first order on $H^1_v(\Omega_\delta)$

\begin{lem}
  For any $\delta$, there exists $C>0$ and $h_0$ such that, for any $h< h_0$ and any $v,w \in H^1_v(\Omega_\delta)$
  \[
    \left | q_h(v,w)-a_h(v,w) \right |\,\leq\,C\cdot h\cdot a_h(v)^{\und}a_h(w)^{\und}
  \]
\end{lem}

The proof is the same as the proof of Lemma \ref{lem:asysep}. Details are left to the reader.

Following the proof in the Dirichlet case, we introduce the abstract Hilbert space $\ell^2(\N, \Hcal_\delta)$ by formally identifying
$e_k$ with the canonical Hilbertian basis of $\ell^2(\N)$. This enables us to define an operator $\Acal_h$ that coincides  with the previous
one on the modes $k\geq 1$. On the zeroth mode, we have
\[
  A_{0,h}v_0\,=\,-h^2 \frac{1}{L}(Lv')'. 
\]

In order to prove the analogue of Proposition \ref{prop:ncline}, we thus only need to study the concentration estimate for this operator.
Actually, using that
\[
   A_{0,h}v_0\,=\,-h^2 v'' \,-\,h^2\frac{L'}{L}v',
\]
and the fact that we work in $I_\delta$, the following lemma is enough.

\begin{lem}
  Let $P_h$ be the Dirichlet realization of $-h^2v''$ in $L^2(I_\delta)$.
  For any $\eps$, and any compact set $K \subset ]0,+\infty[$, there exist a constant $C$ and $h_0$ such that
  if $h\in ]0,h_0]$, $E\in K$ and $v\in H^1_0(I_\delta)$
  \[
    \| v \|_{H^1_{P_h}} \,\leq\, \frac{\eps}{h} \| (P_h-E)v \|_{H^{-1}_{P_h}}\,+\,C \| hv'\|_{L^2(I_\delta)}.
  \]
\end{lem}
\begin{proof}
  We follow the proof of Proposition \ref{prop:ncline}. The spectrum of $P_h$ is
  given by $\left\{ \frac{k^2\pi^2}{\ell^2},~k\geq 1 \right\}$ where 
  $\ell = 2-2\delta$ is the length of the interval $I_{\delta}$.
  We observe that if $k$ is such that 
  $\frac{h^2k^2\pi^2}{\ell^2} \in [a,b]\subset \op 0,+\infty \cl$ then $k$ is of order $1/h$
  so that the distance between two consecutive eigenvalues is also of order $1/h$ as desired (see Equation (\ref{eqn:dist-est})). Next, we need a description of the
  quadratic form $\Bcal_h$. The matrix element of the latter can be explicitly described
  since the eigenfunctions of $P_h$ are known:
  They are given by $x\mapsto \sin \left(\frac{k\pi}{\ell}(x+1-\delta) \right)$. In the interval $[a,b]$, the relevant $k$ is of order $1/h$ so that
  the asymptotics of $\Bcal$ are easily computed. 
  All of the arguments in the proof of Proposition \ref{prop:ncline}
  are then seen hold in the present case.
\end{proof}

By following the proof of Theorem \ref{thm:limits}, we are able to prove the following proposition.  We recall that a `threshold'  is an element of the set
\[
  \left\{ \frac{k^2\pi^2}{L^2(0)},~k\in \N
  \right\}.
\]

\begin{prop}
\label{prop:controlestimate}
  For any $\delta$, any $\eps$, and any interval $[a,b]$ that does not contain any threshold,
  there exists a constant $C$ and $h_0$ such that, for any $h<h_0$ and for any eigenfunction $u_h$ of $q_h$ whose eigenvalue $E_h$ is in $[a,b]$,
  we have
  \begin{equation}
  \label{eqn:central} 
    \| u_h\|_{L^2(\Omega_\delta)}\,
    \leq\, 
    \eps\, \|u_h\|_{L^2(\Omega)}\,
    +\,
    C\| h\, \partial_x u_h\|_{L^2(\Omega)}. 
  \end{equation}
\end{prop}

\begin{proof}
  We fix a cutoff function $\chi$ that is identically $1$ in $I_\delta$ and identically $0$ outside $I_{\frac{\delta}{2}}$.
  The previous lemma implies that the analogue of Proposition \ref{prop:ncsum} holds true 
  \[
    \| \chi u\|_{H^1_{\Acal_h}}~
    \leq~
    \frac{\eps}{h}\, \|(\Acal_h-E_h)(\chi u_h)\|_{H^{-1}_{\Acal_h}}\,
    +\,
    C\, \|h\, D'(\chi u_h)\|_{L^2(\Omega_{\delta/2})}
  \]
  where all the norms are taken on $I_{\delta/2}$.
  Using the same arguments as in the proof of Theorem \ref{thm:limits}, we show first that there exists some $M$ (that depends only on $\delta,\,a,$ and $b$)
  such that
  \[
    \|(\Acal_h-E_h)(\chi u_h)\|_{H^{-1}_{\Acal_h}}\,\leq\,M\|u_h\|_{L^2(\Omega)}.
  \]
  and then that there exists another constant $M$ such that 
\[
\|h\, D'(\chi\, u_h)\|_{L^2(\Omega_{\delta/2})}\,\leq\,M\left( \|h\, \partial_x u_h\|_{L^2(\Omega)}\,+\,h\, \|u_h\|\right ).
\]
The claim follows.  
\end{proof}

As before, we will use the preceding proposition to prove that if an eigenvalue branch converges to a limit that is not a
threshold, then its mass must concentrate away of $\Omega_\delta$.
Our aim will then be to obtain a contradiction.

So for the remainder of this appendix, we make the following assumption:

\begin{ass}\label{ass:notthreshold}
\label{ass:no-theshold}
 We assume that  $(u_h,E_h)$ is  an analytic eigenbranch of $q_h$ such that
  \begin{equation}\label{eq:notthreshold}
    \lim_{h\rightarrow 0} E_h~
    \notin~
    \left\{ \frac{k^2\pi^2}{L^2(0)},~k \in \N \right \}.
  \end{equation}
\end{ass}

Using the Feynman-Hellmann 
formula (\ref{eqn-Feynman-Hellmann}) as in  the proof of Theorem 
\ref{thm:limits}, we find  that there exists a 
subset $\Hbb \subset \R^*$ with accumulation point $0$
such that for each $\epsilon>0$ there exists $h_0$ 
so that if $h \in \op 0, h_0\cl \cap \Hbb$, then
\begin{equation}
\label{eqn:integrability2}
\|h \cdot \partial_x u_h\|_{L^2(\Omega)}\,
\leq\,
\eps\, \|u_h\|_{L^2(\Omega)}.
\end{equation}
Assumption \ref{ass:no-theshold} implies that there exists an 
interval an $[a,b]$ not containing a threshold so that if $h$
is sufficiently small, then $E_h \in [a,b]$. Therefore, 
we may apply Proposition \ref{prop:controlestimate} to the 
sequence $(u_h)_{h \in \Hbb}$ and  obtain

\begin{prop}
\label{prop:coninwings}
 Under assumption \ref{ass:notthreshold}, there exists a subsequence $\Hbb$ going to $0$ such that, for any $\eps,\delta>0$, 
there exists $h_0$ such that if $h \in  \op 0,h_0 \cl \cap \Hbb$, then
\begin{equation}
\label{eqn:no-threshold}
    \|u_h\|_{L^2(\Omega_\delta)}\,\leq \, \eps\, \|u_h\|_{L^2(\Omega)}.
\end{equation}
\end{prop}


\subsection{Outside of the central region}

In this section, we assume that $(u_h)_{h\in \Hbb}$ is a sequence 
provided by Proposition \ref{prop:coninwings}.
The proof of Theorem \ref{thm:limits-Neumann} will be
completed by showing that estimate 
(\ref{eqn:no-threshold}) cannot hold true.
That is, Assumption \ref{ass:no-theshold} will be contradicted.

We want to control $u_h$ near $x= \pm 1$.
These two problems are completely equivalent so that we will only treat the case 
where $x$ is near $-1$.
In order to simplify the notation, we make the change of variable  $x \leftarrow x+1$ so that
$x$ is now near $0$, and $L(x)\,=\,\sqrt{x(2-x)}$. We set
\[
  U_{\delta}~
  =~
  \left\{\, (x,y),~~x\in \op 0,\delta \cl,~ 0< y< L(x)~ \right\}.
\]
In other words, $U_{\delta}$ is the connected component of 
$\Omega \setminus \Omega_{\delta}$ that lies to the left of the central region $\Omega_{\delta}$.

Since $u_h$ is an eigenvalue, we will use several times the following estimate that we call the $H^1_h$ bound:
\[
  \|h\, \px u_h\|^2_{L^2(\Omega)}\,+\,\|\py u_h\|^2\,\leq\, C\|u_h\|_{L^2(\Omega)}.
\]

We recall the transversal Fourier decomposition of $u_h$: 
\[
u_h \,=\,v_h\otimes \un \,+\,\sum_{k\geq 1} u_{h,k}\otimes e_k.
\]
The function $v_h\otimes \un$ is the (transversal) zeroth mode of $u_h$, and we define
$u_h^{\bot}\,=\,u_h-v_h\otimes \un$.
Our aim is to prove there exists some $\delta$ so that, for $h$ small enough, 
all of the mass of $u_h$ cannot
lie in $U_\delta$. We will use three different estimates to achieve our goal:
\begin{itemize}
\item An estimate for $u^\bot_h$ in $U_\delta$ that follows from a Poincaré inequality. 
(See Lemma \ref{lem:u-perp}.)
\item An estimate for $v_h$ in  the interval $[0,\eps h]$ that follows from the $H^1_h$ bound.
(See Corollary \ref{coro:H1-bound}.)
\item An estimate for $v_h$ in the interval $[\eps h, \delta]$ that is obtained using standard
  methods in the study of one-dimensional Schrödinger (or Sturm-Liouville) equations.
  (See Proposition \ref{prop:boundonv}).
\end{itemize}
It will be possible to choose the parameters $\eps$ and $\delta$, 
and we will always assume that
$h$ is chosen so that $\eps h < \delta$. We also fix some (small) upper bound $\delta_0$ so as to have
estimates that depend on $\delta_0$ and not on $\delta < \delta_0$.

\medskip

\subsubsection{The estimate for $u_h^{\perp}$}

\begin{lem}
\label{lem:u-perp}
  There exists a function $f$ that goes to $0$ when $\delta$ goes to zero so that if 
  $h\leq h_0$, then 
  \[
   \| u_h^{\perp}\|_{L^2(U_\delta)}\,
    \leq\, 
    f(\delta)\, \|u_h\|_{L^2(\Omega)}.
  \]
\end{lem}
\begin{proof}
  For any $k\geq 1$, we have
  \[
    \int_0^\delta |u_{h,k}(x)|^2 L(x) dx \,\leq L^2(\delta) \int_{0}^\delta \frac{k^2\pi^2}{L^2(x)}|u_{h,k}(x)|^2 L(x) dx.
  \]
  It follows that
  \[
    \| u_h^\perp \|^2_{L^2(U_\delta)}~
    \leq~
    L^2(\delta)\,
    \| \py u_h
    \|^2_{L^2(U_\delta)}~
    \leq~
    C\cdot {L^2(\delta)}\,
    \| \py u_h\|^2_{L^2(\Omega)}.
  \]
\end{proof}

\subsubsection{The estimate for $v_h$ in $[0,\eps h]$} \ \\

In this regime, we only need the fact that $u_h$ is in $H^1(\Omega)$. 
Thus, we will supress the dependence of $u$ on $h$ here.

\begin{lem}
  There exists a constant $C$ such that,
  for any $u\in H^1(\Omega)$, for any $\alp,\beta$ such that
  $\alpha < \beta < \frac{1}{2}$ we have
  \[
    \int_0^\alpha |v(x)|^2\,L(x)dx\,\leq\, C \left[
      \left(\frac{\alpha}{\beta}\right)^{\frac{3}{2}}\|u\|^2_{L^2(\Omega)}
      \,+\,\alpha^\td\beta^\und \|\px u\|^2_{L^2(\Omega)}\,+\,\alpha \|\py u\|_{L^2(\Omega)}^2\right ],
  \]
  where $v\otimes \un$ is the zeroth-mode of $u$.
\end{lem}

\begin{proof}
  By density, we may assume that $u\in C^\infty(\bar{\Omega})$.
  Using the transversal Fourier decomposition, we have
  \[
    \px u \,=\,v'\otimes \un \,+\,\sum_{k\geq 1} u_k'\otimes e_k -\frac{yL'}{L}\py u.
  \]
  Since $yL'/L \in L^2(\Omega)$ and $|\nabla u|$ is bounded, it follows that
  \[
    \px u \,+\,\frac{yL'}{L}\py u \in L^2(\Omega),
  \]
  so that
  \[
    v'(x)\,=\,\frac{1}{L(x)}\int_0^{L(x)} \left( \px u(x,y)\,+\,\frac{yL'(x)}{L(x)}\py u(x,y)\right) \, dy.
  \]
  For a function $w\in L^2(\Omega)$, we introduce the notation
  \[
    N_{w}(x)\,=\,\left( \int_0^{L(x)} |w(x,y)|^2\, dy\right)^\und.
  \]
  By construction, for any $\delta>0$,
  \[
    \int_0^\delta |N_w(x)|^2\, dx\,\leq\, \|w\|_{L^2(\Omega)}^2.
  \]
  Using the Cauchy-Schwarz inequality on the definition of $v'$ and the behaviour at $0$ of $L$, we get that there
  is a constant $C$ such that
  \[
    \forall x\leq \frac{1}{2},~~ |v'(x)|\,\leq\,C\left[ x^{-\frac{1}{4}}N_{\px u}(x)\,+\,x^{-\frac{3}{4}}N_{\py u}(x)\right ].
  \]
  Integrating this inequality on an interval $[x,y]$, and using Cauchy-Schwarz again, we obtain
  \[
    |v(y)-v(x)|\,\leq\, C\left[ (y^{\frac{1}{2}}-x^{\frac{1}{2}})^{\und}\| \px u\|_{L^2(\Omega)}\,+
      \,(x^{-\frac{1}{2}}-y^{-\frac{1}{2}})^\und \partial_y u\|_{L^2{(\Omega)}}\right ].
  \]
  We square, use Young's inequality, and remove the obviously negative terms to obtain that
  for each $x<y$
  \[
    |v(x)|^2\,\leq\,C\left[ |v(y)|^2\,+\,y^\und\| \px u\|^2\,+\,x^{-\und}\| \py u\|^2 \right].
  \]
  We now fix $\alpha<\beta$, multiply by $L(y)$ and integrate the preceding over $y\in [\beta,2\beta]$.
  Using that
  \[
    \int_\beta^{2\beta} L(y)\,dy \asymp \int_\beta^{2\beta} y^{\und} \,dy \asymp \beta^{\td},
  \]
  we obtain
  \[
    \forall x\in (0,\alpha),~~|v(x)|^2 \,\leq\,C\left[ \beta^{-\td}\|u\|^2_{L^2(\Omega)}\,+\,\beta^{\und} \|\px u\|_{L^2(\Omega)}^2
      \,+\,x^{-\und}\| \py u\|_{L^2(\Omega)}^2 \right ].
  \]
  We finally multiply by $L(x)$ and integrate over $x\in  \op 0,\alpha \cl$.
\end{proof}

We obtain the following corollary for $v_h$.
\begin{coro}
\label{coro:H1-bound}
  There exists a constant $C$ such that, for any $\eps\in \op 0,1 \cl$ and $h$ small enough we have
  \[
    \int_0^{\eps h}|v(x)|^2L(x)\,dx\,\leq\, C\eps\, \|u_h\|^2.
  \]
\end{coro}
\begin{proof}
  We let $\beta = h$ and $\alp = \eps h$ to obtain
  \[
    \int_0^{\eps h}|v(x)|^2L(x)\,dx~
    \leq~
      \eps^{\frac{3}{2}}\,
      \|u_h\|^2_{L^2(\Omega)}
      \,+\,
      \eps^\td\,
      \|h\, \px u_h\|^2_{L^2(\Omega)}\,
      +\,
      \eps h\,
      \|\py u_h\|_{L^2(\Omega)}^2.
  \]
  The claim follows using the eigenvalue $H^1_h$ bound on $u_h$.
\end{proof}

We now move on to study $v_h$ on $[\eps h, \delta]$.

\subsubsection{The estimate for $v_h$ in $[\eps h,\delta]$}

Fix some $\delta_0 \,<\,2$, and let $\phi \in \Cinf_0(\op 0, \delta\cl)$.
Integration against the eigenequation gives
$$
h^2\int_{\Omega} \px u_h(x,y) \phi'(x)\, dxdy\,
=\,
E_h\int_{\Omega}u_h(x,y) \phi(x)\, dxdy\,
=\,
E_h\int_{0}^{\delta_0} v_h(x)\phi(x) \, L(x) dx.
$$
Now, recall that 
\[
\px u_h \,=\,v_h'\otimes \un\,+\,\sum_{k\geq 1} u_{h,k}'\otimes e_k\,-\,\frac{yL'}{L}\py u_h,
\]
and integrate this equation against $\phi'$ to find 
\[
h^2\int_{0}^{\delta_0} v_h'(x)\phi'(x) \, L(x)dx -E_h\int_{0}^{\delta_0} v_h(x)\phi(x) \, L(x) dx\,=\,h^2\int_{0}^{\delta_0} r_h(x) \phi'(x) dx, 
\]
where we have set 
\[
r_h(x)\,=\,\int_0^{L(x)} \frac{yL'(x)}{L(x)}\py u_h dy.
\]

We divide by $h^2$ and make an integration by parts on the RHS to find
that 
\[
  \int_{0}^{\delta_0} v_h'(x)\phi'(x) \, L(x)dx -\frac{E_h}{h^2}\int_{0}^{\delta_0} v_h(x)\phi(x) \, L(x) dx
  \,=\,-\int_{0}^{\delta_0} r_h'(x)  \phi(x) \,dx
\]
 
It follows that, on the interval $(0, \delta_0)$, $v_h$ satisfies the following Sturm-Liouville equation:
\[
  -\frac{1}{L}(Lv_h')' -\frac{E}{h^2} v_h \,=\,-\frac{1}{L}r_h'.
\]

We first make the change of dependent variable $w=L^\und v$. The equation becomes
\begin{equation}
\label{eq:eqinhom}
  -w''\,-\,\frac{E}{h^2}w\,+\,\frac{1}{2}L^{-\und}\left( L'L^{-\und}\right)'w\,=\,-L^{-\und}r_h'.
\end{equation}

In order to solve this equation, we first study the associated homogenous equation:
\[
  -w''\,-\,\frac{E}{h^2}w\,+\,\frac{1}{2}L^{-\und}\left( L'L^{-\und}\right)'w\,=\,0. 
\]

We can write $L$ as
\[
  L(x)\,=\, \sqrt{2x}\,(1\,+\,\tilde{\ell}(x)),
\]
where $\tilde{\ell}$ is smooth on $[0,\delta_0]$. It follows that we can write 
\begin{equation}{\label{eq:hom}}
  \frac{1}{2}L^{-\und}\left( L'L^{-\und}\right)'(x)\,=\,-\frac{3}{16x^2}(1\,+\,x\ell(x)),
\end{equation}
where $\ell$ is smooth on $[0,\delta_0]$.  

We now perform the change of independent variable $x= \frac{\sqrt{E}}{h} z$.
Setting $W(z)=w(x)$, we are led to the following homogeneous equation on 
$\left[\eps\sqrt{E_h}, \frac{\delta_0\sqrt{E}}{h} \right]$:
\[
  W''\,+\,W\,+\,\frac{3}{16z^2}W\,=\,f_h(z)\, W.
\]
We set $J_h\,=\,\left[\eps\sqrt{E_h}, \frac{\delta_0\sqrt{E}}{h} \right]$ and observe that, for any $h\in \Hbb$, we have
$J_h\subset \bar{J}_h\,= [\eps\sqrt{a},Z_h]$ with $Z_h=\frac{\delta_0\sqrt{b}}{h}$.
The function $f_h$ satisfies the following uniform bound:
\begin{gather}
\nonumber
\exists M,~~\forall h\leq h_0,~\forall E\in [a,b],~\forall z \in \bar{J}_h,\\
\label{eqn:fbound}
|f_h(x)|\,\leq\, M\frac{h}{z}.
\end{gather}
We solve the preceding equation as a perturbation of the equation
\begin{equation}
\label{eqn:normalized-ODE}
  W''\,+\,W\,\,+\frac{3}{16z^2}W\,=\,0.
\end{equation}
It is standard that any solution to equation (\ref{eqn:normalized-ODE}) 
on the half-line is asymptotic
to a linear combination of $\cos x$ and $\sin x $ near $+\infty$.
We define $\psi_c$ to be the solution that is asymptotic to cosine and $\psi_s$ 
to be the solution that is asymptotic
to sine. We now apply the variation of constants method: We define an 
operator $K$ on $C^0(\bar{J}_h)$ by
\[
  K[W](z)\,=\, \frac{1}{2}\int_z^{Z_h}f_h(\zeta)W(\zeta)\left[ \psi_c(z)\psi_s(\zeta) - \psi_s(z)\psi_c(\zeta) \right]\,d\zeta,
\]
and observe that the function $W\,+\,K[W]$ is a solution to the unperturbed equation
(\ref{eqn:normalized-ODE}).
It follows that a basis of solutions is given by $\{\phi_s, \phi_c\}$ where
\[
  \phi_s\,=\,\left[ \id \,+\, K\right]^{-1}\psi_s \ \mbox{\ and \ } \ \phi_c\,=\,\left[ \id \,+\, K\right]^{-1}\psi_c,
\]
provided that $\id + K$ is invertible. Since $\psi_c$ and $\psi_s$ 
are uniformly bounded on $[\eps\sqrt{a},+\infty)$, the uniform bound on $f_h$ in 
(\ref{eqn:fbound}) implies 
that there exists some constant $C$ that depends on $a,b,\eps,\delta_0$ such 
that for each $h \in \Hbb$ the operator norm of $K$ is at most $C h\,|\ln h|$.
This ensures the invertibility of $\id + K$ for small $h$, 
and so $\phi_s$ and $\phi_c$
are well-defined.

By unwinding the change of independent variable, we find that
\[
\left\{
  x\mapsto \phi_{s} \left(\frac{x\sqrt{E_h}}{h} \right),~
    x\mapsto \phi_{c} \left(\frac{x\sqrt{E_h}}{h} \right)
  \right\}
\]
is a basis of solutions to the homogenous equation (\ref{eq:hom}).
Since the Wronskian of these two solution is $\frac{2\sqrt{E_h}}{h}$ we find that
the following function $w_p$ solves (\ref{eq:eqinhom}) on the interval $[\eps h, \delta]$
for any $\delta$ such that $\delta \leq \delta_0$.
\[
  \begin{split}
    w_p(x)\,
    =\,& 
    \frac{h}{2\sqrt{E}}\phi_c(x) \int_x^{\delta}L^{-\und}(\xi)\, 
    r'(\xi)\,
    \phi_s \left(\sqrt{E}\xi/h \right)\,d\xi\,
    \\
    & 
    -\frac{h}{2\sqrt{E}}\phi_s(x)\,
    \int_x^{\delta}L^{-\und}(\xi)\,
    r'(\xi)\,
    \phi_c \left(\sqrt{E}\xi/h \right)\,d\xi.
  \end{split}
\]
We perform an integration by parts in the integrals. The boundary terms at $x$ cancel out and
the boundary terms coming from $\delta$ contribute to a solution of the homogenous equation.
It follows that the sum of the following four terms is a solution to the equation:
\[
  \begin{split}
    w_p(x)\,
    = & 
    -\frac{h}{2\sqrt{E}}\phi_c(x) 
    \int_x^{2\delta}
    \left(
    L^{-\und}\right)'(\xi)\,
    r(\xi)\,
    \phi_s \left(\sqrt{E}\xi/h \right)\,d\xi\\\
    & - \phi_c (x) \int_x^{2\delta}L^{-\und}(\xi)r(\xi)
    \phi_s'\left(\sqrt{E}\xi/h \right)\,d\xi\\
    &  + \frac{h}{2\sqrt{E}}\phi_s(x) \int_x^{2\delta}\left( L^{-\und}\right)'(\xi)\,
    r(\xi)\,
    \phi_c \left(\sqrt{E}\xi/h \right)\,d\xi \\
    & +  \phi_s(x) \int_x^{2\delta}
    L^{-\und}(\xi)\,
    r(\xi)\,
    \phi_c'\left(\sqrt{E}\xi/h \right)\,d\xi.
  \end{split}
\]
We denote these four terms by $w_{p,i}$ for $i=1,2,3,4$. In order to bound these terms, we use
the following bound, obtained by applying the Cauchy-Schwarz inequality to the definition of $r$:
There exists $C$ so that 
\[
 |r(\xi)|\,
 \leq\, 
 C\, 
 \xi^{-\frac{1}{4}}\,
 N_{\py u}(\xi). 
\]
We also use that the functions $\phi_{s}$, $\phi_c$, and their derivatives are uniformly bounded on $[\eps\sqrt{a},Z_h)$.

For $i=1$ and $3$ we find that
\[
  \begin{split}
    |w_{p,i}(x)|\,&\leq\, C h\,
    \| \py u\| \left( \int_x^{\delta_0} y^{-3}\right)^{\und}\\
    &\leq\, Ch\, \| \py u\|\, x^{-1}.
  \end{split}
\]

It follows that, for any $\delta \leq \delta_0$
\[
  \begin{split}
    \int_{\eps h}^{\delta} |w_{p,i}(x)|^2 \, dx & \leq \, Ch^2\| \py u\|^2\int_{\eps h}^{\delta} x^{-2}\, dx \\
    &\leq \, C  h \| \py u\|^2.
  \end{split}
\]

We now proceed to estimate $w_{p,i}$ for $i=2$ or $4$.

We proceed as above, using that $\phi_{c,s}'$ is uniformly bounded on $[\eps, \infty)$.
For $i=2,4$ we obtain that 
\[
  \begin{split}
    |w_{p,i}(x)|\,& \leq\,C \| \py u\| \left( \int_{x}^{\delta_0} \xi^{-1}\,d\xi\right)^{\und}\\
    & \leq \, C \| \py u\| |\ln x|^\und.
  \end{split}
\]

It follows that, for any $\delta \leq \delta_0$
\[
  \begin{split}
    \int_{\eps h}^{\delta} |w_{p,i}(x)|^2 \, dx & \leq \, C\| \py u\|^2\int_{\eps h}^\delta |\ln \xi |\, d\xi \\
    &\leq \, C  \| \py u\|^2 \delta\ln|\delta|.
  \end{split}
\]

Finally, the $L^2$ norm of $w_p$ on $[\eps h,\delta]$ satisfies the following.
For any $\eps >0$, there exists a constant $C$ such that, for any $\delta$ and $h$ small enough,
\begin{equation}\label{eq:boundonwp}
  \|w_p\|_{L^2([\eps h, \delta])}\,\leq C\left ( h\,+\,\delta |\ln \delta|)\right)^{\und}\| u\|_{L^2(\Omega)}.
\end{equation}

Now, since $w_p$ is a solution to the same equation as $w_h$ there exist a solution $\phi_h$ to
the homogenous equation such that $w_h\,=\,w_p\,+\,\phi_h$.

We now claim the following 

\begin{lem}
\label{lem:overlap}
  For any $\delta \in \op 0, 1\cl$, there exist constants $C$ and $h_0$ such that, for any $h \leq h_0$ and
  any $\phi$ solution of the homogenous equation
  \[
    \int_{\eps h}^{\delta} \left|\phi\left(\frac{\sqrt{E}}{h}x\right)\right|^2\, dx~
    \leq~
    C \int_{\delta}^{2\delta} \left|\phi\left(\frac{\sqrt{E}}{h}x\right)\right|^2\, dx.
  \]
\end{lem}
\begin{proof}
  Changing variables, it suffices to study the behaviour of
  \[
    \int_{\eps}^{X/h} |\phi(z)|^2\, dz,
  \]
  as $h$ tends to $0$ and $X\leq X_0\,=\,2\sqrt{b}\delta_0$.
  There exists $\alpha,\,\beta $ such that $\phi=\alpha \phi_c\,+\,\beta \phi_s$.
  By construction, on $[\eps,X_0/h]$, we have
  \[
    \begin{split}
      \| \phi_c -\psi_c \|_{C^0}&\leq\, Mh\, |\ln h|\, \|\psi_c\|_{C^0},\\
      \| \phi_s-\psi_s\|_{C^0}&\leq\, Mh\, |\ln h|\, \|\psi_s\|_{C^0}.
    \end{split}
  \]
  and, for $z$ going to infinity, we have
  \[
    \begin{split}
      \psi_c(z)&= \cos z \,+\ O(z^{-1}),\\
      \psi_s(z)&= \sin z \,+\, O(z^{-1}).
    \end{split}
  \]
  Finally, we obtain:
  \[
    \begin{split}
      \int_{\eps}^{X/h} |\phi(z)|^2\, dz & \,=\,
      \int_{\eps}^{X/h} (\alpha\cos z \,+\,\beta \sin z)^2\,dz \\
      & + (|\alpha|^2+|\beta|^2)\, O(\ln|X/h|)\,\\
      & + (|\alpha|^2+|\beta|^2)\, |X/h|\, O(h^2\ln|h|^2).
    \end{split}
  \]
  A direct computation gives that
  \[
    \int_{\eps}^{X/h} (\alpha\cos z \,+\,\beta \sin z)^2\,dz\,=\,\frac{\alpha^2+\beta^2}{2}\frac{X}{h}\,+\,|ab|\, O(1).
  \]
  It follows that there exist a constant $C,$ and $h_0$ so that, for any $X$, $\alpha$, $\beta$, and $h$ small enough we have
  \[
    C^{-1}\cdot \frac{X(\alpha^2\,+\,\beta^2)}{h}\,
    \leq\,\int_{\eps}^{X/h} |\phi(z)|^2\, dz
    \,\leq\, C\cdot \frac{X(\alpha^2\,+\,\beta^2)}{h}.
  \]
  The claim now follows.
\end{proof}

All of these estimates may be combined to give 

\begin{prop}
\label{prop:boundonv}
  For any $\eps$ and $\delta_0$, there exists a constant $C$ such that for each $\delta < \delta_0$
  there exists $h_0$ such that if  $h \leq h_0$, then
\begin{equation}
\label{eqn:middle-regime}
    \int_{\eps h}^{\delta} |v_h(x)|^2 \,L(x)\,dx\,
    \leq\,
    C\,
    \left(\|u\|_{L^2(\Omega_{\delta})}^2\,
    +\,
    \left(h\,+\, \delta \,|\ln \delta|\right)\cdot\|u\|^2_{L^2(\Omega)}\right). 
\end{equation}
\end{prop}

\begin{proof}
  By definition of $w_h$, we have
  \[
    \int_{\eps h}^{\delta} |v_h(x)|^2 \,L(x)\,dx\,
    =\,
    \int_{\eps h}^{\delta}|w_h(x)|^2 \,dx.
  \]
  Since $w_h=w_p+\phi$, we obtain
  \begin{equation*}
  \label{eqn:w_h}
    \int_{\eps h}^{\delta}\,|w_h(x)|^2 \,dx\,
    \leq\, 
    2\left( \int_{\eps h}^{\delta}|\phi_h(x)|^2 \,dx\,
      +\,
      \int_{\eps h}^{\delta}|w_p(x)|^2 \,dx\right).
  \end{equation*}
  The second integral may be estimated using Lemma \ref{lem:overlap}:
  \[
   \int_{\eps h}^{\delta} |\phi_h(x)|^2 \,dx~
   \leq~
    \int_{\delta}^{2\delta} |\phi_h(x)|^2 \,dx\ \,
    \leq\,
    2\left(
      \int_\delta^{2\delta} |w_h(x)|^2 \,dx\ \,+\,\int_\delta^{2\delta} |w_p(x)|^2 \,dx
      \right ).
  \]
  Estimate (\ref{eq:boundonwp}) can be used to estimate each of the integrals involving $w_p$, and 
  so we obtain:
  \[
    \int_{\eps h}^{\delta}|w_h(x)|^2 \,dx\,
    \leq\, 
    C\left(\int_{\delta}^{2\delta} |w_h(x)|^2 \,dx
      \,+\,
      (h\,+\,\delta\ln \delta)\|u\|^2_{L^2(\Omega)}\right).
  \]
  The claim follows since
  \[
    \int_\delta^{2\delta}|w_h(x)|^2 \,dx~
    \leq~
    \int_\delta^{2\delta} |v(x)|^2 \,L(x)dx~
    \leq~
    \|u_h\|^2_{L^2(\Omega_{\delta})}.
  \]
\end{proof}


\subsubsection{Obtaining the contradiction}\hfill \\

We now finish the proof of Theorem \ref{thm:limits-Neumann}. 
Recall that we are assuming for contradiction that the statement of the 
theorem is false, and hence 
Proposition \ref{prop:coninwings} provides us with a sequence $(u_h)_{h \in \Hbb}$
so that that given $\eps,\delta>0$, 
there exists $h_0$ such that if $h \in  \op 0,h_0\cl \cap \Hbb$, then
\begin{equation}
\label{eqn:no-threshold2}
    \|u_h\|_{L^2(\Omega_\delta)}\,\leq \, \eps\, \|u_h\|_{L^2(\Omega)}.
\end{equation}

Fix $\eta>0$. By Proposition \ref{prop:coninwings}, Lemma \ref{lem:u-perp}, Corollary \ref{coro:H1-bound}, there exists $\eps, \delta_0, h_0>0$ so that if $\delta \leq \delta_0$ 
there exists $h_0$ so that if  $h \in \op 0, h_0 \cl \cap \Hbb$, then 
\begin{equation}
\label{est:central}
\|u_h\|_{L^2(\Omega_\delta)}\,\leq \, \eta\, \|u_h\|_{L^2(\Omega)},
\end{equation}
\[
  \| u^\bot_{h}\|_{L^2(U_{\delta})}^2 ~
  \leq~
  \eta\,
  \|u_h\|^2_{L^2(\Omega)}, 
\]
and 
\[
  \int_{0}^{\eps h} |v_h(x)|^2 \,L(x)\,dx~
  \leq~  
  \eta\,
  \|u_h\|^2_{L^2(\Omega)}.
\]
With these choices of $\eps$ and $\delta_0$, 
we apply  Proposition \ref{prop:boundonv} to obtain a 
constant $C$ for which (\ref{eqn:middle-regime}) holds if $h$ is 
sufficiently small. 
Choose $\delta< \delta_0$ so that $C\delta\,|\ln \delta| \leq \eta$.
With this choice of $\delta$ we have 
\[
  \begin{split}
    \| u_h\|^2_{L^2(U_{\delta})}\,
    &=\, 
    \| u_h^{\bot}\|^2_{L^2(U_{\delta})} \,+\,\int_0^{\delta} |v_h(x)|^2 \,L(x)dx\\
    & 
    \leq\,
    \eta \, \|u_h\|^2_{L^2(\Omega)} \,+\,\eta\, \|u_h\|^2_{L^2(\Omega)}\,+\,\int_{\eps h}^{\delta} |v_h(x)|^{2} \,L(x)dx \\
    &\leq\,
    2 \eta\,
    \|u_h\|^2_{L^2(\Omega)}\,
    +\,C\,
    \|u_h\|^2_{L^2(\Omega_{\delta})}\,+\,(C\cdot h \,+\,\eta)\|u_h\|^2_{L^2(\Omega)}\\
    & \leq\,
    3 \eta\,
    \| u_h\|^2_{L^2(\Omega)}\,
    +\,
    C\, \|u_h\|^2_{L^2(\Omega_{\delta})}\,
    +\,
    C\cdot h\, \| u_h\|^2_{L^2(\Omega)}.
  \end{split}
\]
The last term is less than $\eta\, \| u_h\|^2_{L^2(\Omega)}$ if $h < \eta/C$. 
To bound the second term, we apply Proposition \ref{prop:controlestimate}:
There exists $C'$ so that if $h$ is sufficiently small, then
\[
  \|u_h\|^2_{L^2(\Omega_{\delta})}~
  \leq~
  \frac{\eta}{C}\, 
  \|u\|_{L^2(\Omega)}^2\,+\,C'\, \|h\, \px u\|^2.
\]
Using these estimates, we find that for 
$h$ small enough we have
\[
  \| u_h\|^2_{L^2(U_{\delta})}~
  \leq~
  5 \eta\, \|u_h\|^2_{L^2(\Omega)}\,+\, C\, C'\, \| h \, \px u\|^2.
\]
Therefore, the integrability condition (\ref{eqn:integrability2})
implies that for $h$ small enough
\[
   \| u_h\|^2_{L^2(U_{\delta})}~
   \leq~
   6 \eta\, \|u_h\|^2_{L^2(\Omega)}.
\]
Recall that $U_{\delta}$ is the component of $\Omega \setminus \Omega_{\delta}$
on the left of the central region $\Omega_{\delta}$. 
The same argument applies to the component of $\Omega \setminus \Omega_{\delta}$
on the right of the central region. Thus, in sum we have
\[
 \| u_h\|^2_{L^2(\Omega\setminus \Omega_{\delta})}~
 \leq~
 12 \eta\, \|u_h\|^2_{L^2(\Omega)}.
\]
By combining this with (\ref{est:central}) we find that 
$$
\| u_h\|^2_{L^2(\Omega)} \leq 13 \eta\, \| u_h\|^2_{L^2(\Omega)}
$$
which is absurd if $\eta < 1/13$.


\end{document}